\documentclass[11pt]{amsproc}

\usepackage{amscd, amsmath, amsthm, amssymb, mathtools, mathrsfs, stmaryrd}
\usepackage{ascmac}
\usepackage{comment}
\usepackage{bm, bbm}
\allowdisplaybreaks

\usepackage{graphicx}
\usepackage[top=30mm, bottom=30mm, left=25mm, right=25mm]{geometry}

\usepackage{hyperref}

\usepackage{tikz}
\usetikzlibrary{intersections, calc, arrows.meta}

\usepackage{here}
\usepackage{time}
\usepackage[abbrev]{amsrefs}

\usepackage{xcolor}
\usepackage[capitalize,nameinlink,noabbrev,nosort]{cleveref}
\hypersetup{
	colorlinks=true,       
	linkcolor=brown,          
	citecolor=brown,        
	filecolor=brown,      
	urlcolor=brown,           
}

\makeatletter
\@namedef{subjclassname@2020}{\textup{2020} Mathematics Subject Classification}
\makeatother

\newtheorem*{theorem*}{\hspace{-6.3mm}\textbf{Theorem}}  

\newtheorem{theoremcounter}{Theorem Counter}[section]

\theoremstyle{remark}

\theoremstyle{definition}
\newtheorem{definition}[theoremcounter]{Definition}
\newtheorem{example}{Example}

\theoremstyle{plain}
\newtheorem{lemma}[theoremcounter]{Lemma}
\newtheorem{proposition}[theoremcounter]{Proposition}
\newtheorem{corollary}[theoremcounter]{Corollary}

\newtheorem{theorem}[theoremcounter]{Theorem}

\numberwithin{equation}{section}

\newcommand{\Z}{\mathbb{Z}}
\newcommand{\Q}{\mathbb{Q}}
\newcommand{\R}{\mathbb{R}}
\newcommand{\C}{\mathbb{C}}

\newcommand{\J}{\mathbb{J}}

\newcommand{\dd}{\mathrm{d}}
\newcommand{\bbH}{\mathbb{H}}
\newcommand{\calD}{\mathcal{D}}

\newcommand{\calQ}{\mathcal{Q}}
\newcommand{\calR}{\mathcal{R}}

\newcommand{\merM}{M^\mathrm{mer}}

\DeclareMathOperator{\ImNew}{Im}
\renewcommand{\Im}{\ImNew}
\DeclareMathOperator{\ReNew}{Re}
\renewcommand{\Re}{\ReNew}
\DeclareMathOperator{\divNew}{div}
\renewcommand{\div}{\divNew}

\DeclareMathOperator{\SL}{SL}
\DeclareMathOperator{\PSL}{PSL}
\DeclareMathOperator{\GL}{GL}

\DeclareMathOperator{\ord}{ord}
\DeclareMathOperator{\Div}{Div}
\DeclareMathOperator{\End}{End}

\newcommand{\pmat}[1]{\begin{pmatrix}#1\end{pmatrix}}

\newcommand{\smat}[1]{\bigl(\begin{smallmatrix}#1\end{smallmatrix}\bigr)}

\begin{document}

\title[]{Hecke equivariance of the divisor map}

\author[]{Daeyeol Jeon}
\address{Department of Mathematics Education, Kongju National University, Gongju, 32588, Republic of Korea}
\email{dyjeon@kongju.ac.kr}

\author[]{Soon-Yi Kang}
\address{Department of Mathematics, Kangwon National University, Chuncheon, 24341, Republic of Korea}
\email{sy2kang@kangwon.ac.kr}

\author[]{Chang Heon Kim}
\address{Department of Mathematics, Sungkyunkwan University, Suwon, 16419, Republic of Korea\\ School of Mathematics, Korea Institute for Advanced Study (KIAS), Seoul 02455, Republic of Korea
}
\email{chhkim@skku.edu}

\author[]{Toshiki Matsusaka}
\address{Faculty of Mathematics, Kyushu University, Motooka 744, Nishi-ku, Fukuoka 819-0395, Japan}
\email{matsusaka@math.kyushu-u.ac.jp}





\maketitle

\begin{abstract}
	We study the multiplicative Hecke operators acting on the space of meromorphic modular forms, and show that the divisor map to divisors on $X_0(N)$ is a Hecke equivariant map. As applications, we investigate the divisor sum formula of Bruinier--Kohnen--Ono and more general Rohrlich-type divisor sums for polyharmonic Maass forms, discussing several implications for the Hecke action and its relation to the self-adjointness of the Hecke operators.
\end{abstract}


\section{Introduction}\label{sec:Intro}

Hecke operators are linear operators acting on the space of modular forms. A holomorphic function $f: \bbH \to \C$ defined on the upper-half plane $\bbH \coloneqq \{\tau \in \C : \Im(\tau) > 0\}$ is called a (\emph{weakly holomorphic}) \emph{modular form} of weight $k \in \Z$ if it satisfies the functional equation $f(\frac{a\tau+b}{c\tau+d}) = (c\tau+d)^k f(\tau)$ for any $\smat{a & b \\ c & d} \in \SL_2(\Z)$ and admits a Fourier series expansion of the form
\[
	f(\tau) = \sum_{m = m_0}^\infty c_f(m) q^m
\]
for some $m_0 \in \Z$ and coefficients $c_f(m) \in \C$, where $q \coloneqq e^{2\pi i\tau}$. For a positive integer $n$, the \emph{Hecke operator} $T(n)$ is defined, albeit in an ad hoc manner, as follows:
\[
	(f|_k T(n))(\tau) = n^{1-k/2} \sum_{m \in \Z} \left(\sum_{0 < d \mid (m,n)} d^{k-1} c_f \left(\frac{mn}{d^2}\right) \right) q^m.
\]
When considering cusp forms, simultaneous eigenfunctions of the Hecke operators play a crucial role in number theory. However, if the Fourier expansions of modular forms contain terms with negative exponents, such eigenfunctions cannot exist. Nevertheless, numerous studies have investigated the action of Hecke operators on weakly holomorphic modular forms from various perspectives. Notably, research by Borcherds~\cite{Borcherds1992} and Asai--Kaneko--Ninomiya~\cite{AsaiKanekoNinomiya1997} on the elliptic modular $j$-function has been especially influential.

In what follows, we focus on a particular study by Bruinier--Kohnen--Ono~\cite{BruinierKohnenOno2004}, which investigate the relationship between special values and Fourier coefficients of modular forms. To begin with, we define the \emph{elliptic modular $j$-function} $j_1(\tau)$ to be the unique weakly holomorphic modular form of weight $0$ whose Fourier expansion is $j_1(\tau) = q^{-1} + 24 + O(q)$. Given a positive integer $n$, we then define
\[
	j_n(\tau) \coloneqq (j_1|_0T(n))(\tau) = q^{-n} + 24 \sigma_1(n) + O(q),
\]
which is again a weight $0$ weakly holomorphic modular form. To broaden the scope, we also consider \emph{meromorphic modular forms}, which generalize weakly holomorphic modular forms by allowing poles in $\bbH$. Within this framework, Bruinier, Kohnen, and Ono established the following theorem.

\begin{theorem*}[Bruinier--Kohnen--Ono~\cite{BruinierKohnenOno2004}]
	Let $\Theta \coloneqq \frac{1}{2\pi i}\frac{\dd}{\dd \tau} = q \frac{\dd}{\dd q}$. For a positive integer $n$ and a meromorphic modular form $f$ (of weight $k \in \Z$), we have
	\begin{align}\label{thm:BKO}
	\begin{split}
		(j_n, f)_\mathrm{BKO} &\coloneqq \sum_{z \in \SL_2(\Z) \backslash \bbH} \frac{\mathrm{ord}_z(f)}{|\PSL_2(\Z)_z|} j_n(z) + 24 \sigma_1(n) \mathrm{ord}_{i\infty}(f)\\
			&= - \mathrm{Coeff}_{q^n} \left(\frac{\Theta f}{f}\right),
	\end{split}
	\end{align}
	where $\mathrm{ord}_z(f)$ is the order of $f$ at $\tau = z$ and $\mathrm{Coeff}_{q^n}(f)$ denotes the $n$th Fourier coefficient of $f(\tau)$.
\end{theorem*}

For instance, suppose that $f(\tau)$ is the Eisenstein series $E_4(\tau)$ of weight $4$. This function has a unique simple zero at $\tau = \omega = e^{2\pi i/3}$ on $\SL_2(\Z) \backslash \bbH$ and has no zero or pole at the cusp $i\infty$. Using this fact, we obtain
\[
	(j_n, E_4)_\mathrm{BKO} = \frac{1}{3} j_n(\omega) = -\mathrm{Coeff}_{q^n} \bigg(240q - 53280q^2 + 12288960q^3 + \cdots \bigg),
\]
which, in particular, yields $j_1(\omega) = -720$. 

The motivating question of this study is as follows. Noting from the definition that $j_n = j_1|_0T(n)$, the theorem above provides an explicit formula for $(j_1|_0T(n), f)_\mathrm{BKO}$. In this context, if we interpret $(j_n, f)_\mathrm{BKO}$ as the pairing between a weight $0$ modular form and a meromorphic modular form, does there exist a suitable Hecke operator $T^*(n)$ acting on $f$ such that $(j_1|_0 T(n), f)_\mathrm{BKO} = (j_1, f|T^*(n))_\mathrm{BKO}$?

In the example above, if we define
\begin{align*}
	(E_4|T^*(2))(\tau) &= E_4(2\tau) E_4 \left(\frac{\tau}{2}\right) E_4 \left(\frac{\tau+1}{2}\right) = E_{12}(\tau) - \frac{36882000}{691} \Delta(\tau),\\
	(E_4|T^*(3))(\tau) &= E_4(3\tau) E_4 \left(\frac{\tau}{3}\right) E_4 \left(\frac{\tau+1}{3}\right) E_4 \left(\frac{\tau+2}{3}\right) = E_{16}(\tau) + \frac{44449152000}{3617} E_4(\tau) \Delta(\tau),
\end{align*}
then a direct calculation shows that
\begin{align*}
	\frac{\Theta (E_4|T^*(2))}{E_4|T^*(2)} = -53280q + \cdots, \qquad \frac{\Theta (E_4|T^*(3))}{E_4|T^*(3)} = 12288960 q + \cdots,
\end{align*}
which imply the desired identity $(j_1|_0 T(n), E_4)_\mathrm{BKO} = (j_1, E_4|T^*(n))_\mathrm{BKO}$ for $n=2,3$. These examples strongly suggest that the ``multiplicative Hecke operators" introduced by Guerzhoy~\cite{Guerzhoy2006} (or, as Guerzhoy himself noted, originally stemming from the work of Gritsenko--Nikulin~\cite{GritsenkoNikulin1996}) provides an answer to this question.

In this article, we divide the above observations into two fundamental theorems and provide more general framework. To clarify this, we briefly introduce several concepts defined in \cref{sec-Hecke-operator}. For a positive integer $N$, let $X_0(N) \coloneqq \Gamma_0(N) \backslash \bbH^*$ be the compact Riemann surface, where $\bbH^* \coloneqq \bbH \cup \Q \cup \{i\infty\}$. The $\Q$-vector space of divisors on $X_0(N)$ is given by
\[
	\Div(X_0(N))_\Q = \left\{ \sum_{[z] \in X_0(N)} n_z [z] : n_z \in \Q, n_z = 0 \text{ for almost all $[z]$}\right\}.
\]
Let $\merM_*(\Gamma_0(N))$ denote the space of integer-weight meromorphic modular forms on $\Gamma_0(N)$. The first main object is the \emph{divisor map} $\div: \merM_*(\Gamma_0(N)) \setminus \{0\} \to \Div(X_0(N))_\Q$ defined by
\[
	\div(f) \coloneqq \sum_{[z] \in X_0(N)} \ord_{\Gamma_0(N)}(f; [z]) [z],
\]
where $\ord_{\Gamma_0(N)}(f; [z]) \in \Q$ is defined in \cref{sec:divisor-map}. The space $\Div(X_0(N))_\Q$ is known to admit an action of the Hecke operator, given by $D \mapsto T(n)D$, as described in Diamond--Shurman's book~\cite[Chapter 5]{DiamondShurman2005}. On the other hand, for $\merM_*(\Gamma_0(N))$, the action of the multiplicative Hecke operator $f|_* T(n)$ has been studied by Guerzhoy~\cite{Guerzhoy2006}. The first theorem asserts that the divisor map is equivariant with respect to the action of the Hecke operators.

\begin{theorem}[\cref{thm:Hecke-equivariance}]\label{main1}
	For any positive integer $n$ coprime to $N$ and a non-zero $f \in \merM_*(\Gamma_0(N))$, we have
	\[
		T(n) \div(f) = \div(f|_* T(n)).
	\]
\end{theorem}

Although this may seem like a very basic statement, given the limited research on multiplicative Hecke operators, it is likely that this claim has not been explicitly stated before. Furthermore, as explained after \cref{thm:Hecke-equivariance}, contrary to its appearance, this statement is not immediately obvious.

Next, for any function $F: X_0(N) \to \C$, we define a map $\calD_F: \Div(X_0(N))_\Q \to \C$ by
\[
	\calD_F(D) \coloneqq \sum_{[z] \in X_0(N)} n_z F(z)
\]
for $D = \sum_{[z] \in X_0(N)} n_z [z] \in \Div(X_0(N))_\Q$. The precise definition of the action of the Hecke operator on the map $F: X_0(N) \to \C$ is given at the end of \cref{sec:Hecke-modualr-cla}, but it can be regarded as the classical Hecke operator of weight $0$. Similarly, as explained at the beginning of \cref{sec:Rohrlich}, the following arises naturally from the definitions.

\begin{theorem}[\cref{prop:divisor-sums}]\label{main2}
	Let $F: X_0(N) \to \C$. For any positive integer $n$ and $D \in \Div(X_0(N))_\Q$, we have
	\[
		\calD_{F} (T(n) D) = \calD_{F|_0T(n)} (D).
	\]
\end{theorem}

As a generalization of $(j_n, f)_\mathrm{BKO}$, for a map $F: X_0(N) \to \C$ and a meromorphic modular form $f \in \merM_*(\Gamma_0(N))$, we consider the sum so-called the \emph{Rohrlich-type divisor sum}, defined by
\[
	\calD_F(\div(f)) = \sum_{[z] \in X_0(N)} \ord_{\Gamma_0(N)}(f; [z]) F(z).
\]
For $N=1$, if we take $F$ as $F(\tau) = j_n(\tau)$ for $\tau \in \bbH$ and $F(i\infty) = 24\sigma_1(n)$, then we get 
\begin{align}\label{eq:Rohrlich-BKO}
	\calD_F(\div(f)) = (j_n, f)_\mathrm{BKO}.
\end{align}
By combining these two theorems, we obtain the following.

\begin{corollary}\label{cor:BKO}
	Under the notation above, for any positive integer $n$ coprime to $N$, we have
	\[
		\calD_{F|_0T(n)}(\div(f)) = \calD_F(\div(f|_*T(n)).
	\]
	In particular, it implies that $(j_1|_0 T(n), f)_\mathrm{BKO} = (j_1, f|_* T(n))_\mathrm{BKO}$ for any $n$.
\end{corollary}

\begin{proof}
	By two theorems above, we have
	\begin{align*}
		\calD_{F|_0T(n)}(\div(f)) &\overset{\text{\cref{main2}}}{=} \calD_{F}(T(n) \div(f))\\
			&\overset{\text{\cref{main1}}}{=} \calD_{F}(\div(f|_*T(n))).
	\end{align*}
	The expression in~\eqref{eq:Rohrlich-BKO} thus leads to the desired result for $(j_n, f)_\mathrm{BKO}$, which was our original motivation.
\end{proof}

This result is closely related to the self-adjointness of the Hecke operator in the following sense. In our previous work~\cite{JKKM2024}, we studied the Rohrlich-type divisor sums $\calD_F(\div(f))$ associated with a smooth modular form $F(\tau)$ of weight $0$ on $\Gamma_0(N)$. If $F$ has singularities at cusps $\rho$, we assign appropriate values to $F(\rho) \in \C$. We then showed that, under certain analytic conditions related to the convergence of the inner product, this can be expressed in terms of the regularized Petersson inner product. For simplicity, if we assume that the terms $\mathrm{ord}_{\Gamma_0(N)} (f; [\rho])$ and $C_{1,\rho}(f,g), C_{2,\rho}(f,g)$ in the statement of \cite[Theorem 1.1]{JKKM2024} vanish for all cusps $\rho$, (as in the cases discussed in \cref{ex:self-adjoint} and \cref{cor:DIT-KM} later), we obtained the formula
\begin{align}\label{eq:Rohrlich-Petersson}
	\calD_F(\div(f)) = -\frac{1}{2\pi} \langle \Delta_0 F, \log(v^{k/2}|f|)\rangle^{\mathrm{reg}},
\end{align}
where $k$ is the weight of $f$ and $\Delta_0 \coloneqq -v^2(\frac{\partial^2}{\partial u^2} + \frac{\partial^2}{\partial v^2})$ is the hyperbolic Laplacian. In this situation, a direct calculation shows that \cref{cor:BKO} is equivalent to the identity
\[
	\bigg\langle (\Delta_0 F)|_0T(n), \log(v^{k/2} |f|) \bigg\rangle^\mathrm{reg} = \bigg\langle \Delta_0 F, (\log(v^{k/2} |f|)) |_0 T(n) \bigg\rangle^\mathrm{reg},
\]
which expresses the self-adjointness of the Hecke operator $T(n)$.

This article is organized as follows. In \cref{sec-Hecke-operator}, we review the concept of the Hecke algebra following Shimura's book~\cite{Shimura1971} and interpret the three types of Hecke operators introduced in this section as representations of the Hecke algebra. \cref{sec:div-Hecke} is devoted to the proof of \cref{main1}. Finally, in \cref{sec:Rohrlich}, continuing our previous work~\cite{JKKM2024}, we study the Rohrlich-type divisor sums $\calR_{N,m}(s;f)$ defined by the Niebur--Poincar\'{e} series and derive several consequences of the present results. In particular, by proving the $p$-plication formulas (Equations \eqref{eq:p-pricate-non} and \eqref{eq:p-prication}) for the Niebur--Poincar\'{e} series and polyharmonic Maass forms, we explicitly compute the identity given in \cref{cor:BKO} for $\calR_{N,m}(s;f)$ in \cref{cor:L-Hecke-equiv}. As an application, we provide an algebraic shortcut (\cref{cor:DIT-KM}) for an alternative proof of the result by Duke--Imamo\={g}lu--T\'{o}th~\cite{DukeImamogluToth2011}, which we previously established in~\cite[Section 4.6]{JKKM2024}. This new approach simplifies the proof by eliminating the analytic discussion related to the regularized Petersson inner product.

\section{Hecke operators}\label{sec-Hecke-operator}

\subsection{Hecke algebras}

In this section, we provide an overview of a unified approach to understanding Hecke operators on various spaces, viewing them as representations of Hecke algebras. Our exposition follows Shimura's book~\cite[Chapter 3]{Shimura1971}.

\begin{definition}
For a positive integer $N$, we define the set
\[
	\Delta_N \coloneqq \left\{\gamma = \pmat{a & b \\ c & d} \in \mathrm{M}_2(\Z) : \det(\gamma) > 0, (a,N) = 1, N|c \right\},
\]
which coincides with the set $\Delta'$ in Shimura~\cite[(3.3.3)]{Shimura1971} when $\mathfrak{h} = (\Z/N\Z)^\times$ and $t=1$. We then define
\[
	R_0(N) \coloneqq \Z[\Gamma_0(N) \backslash \Delta_N/\Gamma_0(N)],
\]
that is, the free $\Z$-module generated by the double cosets $\Gamma_0(N) \alpha \Gamma_0(N)$ for $\alpha \in \Delta_N$. 
\end{definition}

We introduce a multiplication on $R_0(N)$ that endows it with the ring structure as follows. Given $\alpha, \beta \in \Delta_N$, consider disjoint coset decompositions
\begin{align}\label{eq:coset-decomp-ex}
	\Gamma_0(N) \alpha \Gamma_0(N) = \bigsqcup_{i \in I} \Gamma_0(N) \alpha_i, \quad \Gamma_0(N) \beta \Gamma_0(N) = \bigsqcup_{j \in J} \Gamma_0(N) \beta_j.
\end{align}
Here, the index sets $I$ and $J$ are finite, (see~Diamond--Shurman~\cite[Lemmas 5.1.1 and 5.1.2]{DiamondShurman2005}). For these two double cosets, we obtain a (not necessarily disjoint) coset decomposition of the set $\Gamma_0(N) \alpha \Gamma_0(N) \cdot \Gamma_0(N) \beta \Gamma_0(N) \coloneqq \{a b : a \in \Gamma_0(N) \alpha \Gamma_0(N), b \in \Gamma_0(N) \beta \Gamma_0(N)\}$ as follows:
\begin{align*}
	\Gamma_0(N) \alpha \Gamma_0(N) \cdot \Gamma_0(N) \beta \Gamma_0(N) &= \Gamma_0(N) \alpha \Gamma_0(N) \beta \Gamma_0(N)\\
		&= \bigcup_{(i, j) \in I \times J} \Gamma_0(N) \alpha_i \beta_j.
\end{align*}
As explained in \cite[Section 3.1]{Shimura1971}, the right-hand side can be written as a finite union of double cosets of the form $\Gamma_0(N) \xi \Gamma_0(N)$. 

\begin{example}\label{ex:level1}
	For $N = 1$, let $\Gamma = \SL_2(\Z)$. For $\alpha = \beta = \smat{1 & 0 \\ 0 & 2} \in \Delta_1$, the following coset decomposition can be easily verified:
	\[
		\Gamma \pmat{1 & 0 \\ 0 & 2} \Gamma = \bigsqcup_{i \in I} \Gamma \beta_i,
	\]
	where we set $I = \{\infty, 0, 1\}$ and $\beta_\infty = \smat{2 & 0 \\ 0 & 1}, \beta_0 = \smat{1 & 0 \\ 0 & 2}, \beta_1 = \smat{1 & 1 \\ 0 & 2}$. Then the product is given by
	\[
		\Gamma \alpha \Gamma \cdot \Gamma \beta \Gamma = \bigcup_{(i,j) \in I \times I} \Gamma \beta_i \beta_j.
	\]
	For $S = \{(\infty, \infty), (1,\infty), (0,0), (1,0), (0,1), (1,1)\} \subset I \times I$, a straightforward calculation shows that
	\begin{align*}
		\bigsqcup_{(i,j) \in S} \Gamma \beta_i \beta_j = \Gamma \pmat{1 & 0 \\ 0 & 4} \Gamma, \qquad \Gamma \beta_\infty \beta_0 = \Gamma \beta_\infty \beta_1 = \Gamma \beta_0 \beta_\infty = \Gamma \pmat{2 & 0 \\ 0 & 2} \Gamma.
	\end{align*}
	Thus, in this case, $\Gamma \alpha \Gamma \cdot \Gamma \beta \Gamma$ is the union of the two double cosets $\Gamma \smat{1 & 0 \\ 0 & 4} \Gamma$ (with multiplicity $1$) and $\Gamma \smat{2 & 0 \\ 0 & 2} \Gamma$ (with multiplicity $3$).
\end{example}

We then define the product of $u = \Gamma_0(N) \alpha \Gamma_0(N)$ and $v = \Gamma_0(N) \beta \Gamma_0(N)$ in $R_0(N)$ by
\[
	u \cdot v = \sum_w m(u \cdot v; w) w,
\]
where the sum runs over all double cosets $w = \Gamma_0(N) \xi \Gamma_0(N)$ contained in $\Gamma_0(N) \alpha \Gamma_0(N) \beta \Gamma_0(N)$. The coefficients are given by the multiplicity
\[
	m(u \cdot v; w) \coloneqq  \#\{(i,j) \in I \times J : \Gamma_0(N) \alpha_i \beta_j = \Gamma_0(N) \xi\}.
\]
As shown in~\cite[Section 3.1]{Shimura1971}, the definition of $m(u \cdot v; w)$ is independent of the choice of representatives $\{\alpha_i\}, \{\beta_j\}$, and $\xi$. It is also known that this multiplication satisfies the associative law, which makes $R_0(N)$ a ring. This ring is called the \emph{Hecke algebra}.

\begin{definition}
	For two positive integers $a$ and $d$ such that $a|d$ and $(a,N) = 1$, we define
	\[
		T(a,d) \coloneqq \Gamma_0(N) \pmat{a & 0 \\ 0 & d} \Gamma_0(N) \in R_0(N).
	\]
	Moreover, for any positive integer $n$, we define $T(n) \in R_0(N)$ as the sum of all $\Gamma_0(N) \alpha \Gamma_0(N)$ with $\alpha \in \Delta_N$ and $\det(\alpha) = n$. 
\end{definition}

This $T(n)$ can be explicitly written as
\[
	T(n) = \sum_{\substack{ad = n \\ a \mid d, (a,N) = 1}} T(a,d),
\]
(see~\cite[Lemma 6.5.6]{CohenStromberg2017} for instance). This expression implies that for any prime $p$, we have $T(p) = T(1,p)$. The example computed in \cref{ex:level1} can be written as
\[
	T(2)^2 = T(1,4) + 3 T(2,2) = T(4) + 2T(2,2).
\]
More generally, the following serves as the foundation of the Hecke algebra.

\begin{proposition}[{Shimura~\cite[Theorems 3.34 and (3.3.6)]{Shimura1971}}]\label{prop:Hecke-property}
	The elements $T(n), T(a,d) \in R_0(N)$ satisfy the following:
	\begin{enumerate}
		\item For any integers $m, n$, we have
		\[
			T(m) T(n) = \sum_{\substack{d \mid (m,n) \\ (d,N) = 1}} d T(d,d) T \left(\frac{mn}{d^2}\right).
		\]
		\item The Hecke algebra $R_0(N)$ is a polynomial ring over $\Z$ generated by $T(p)$ for all primes $p$ and $T(q,q)$ for all primes $q \nmid N$, that is,
		\[
			R_0(N) = \Z[T(p), T(q,q) : p,q: \text{primes}, q \nmid N].
		\]
	\end{enumerate}
	In particular, $R_0(N)$ is a commutative ring with the identity $\Gamma \smat{1 & 0 \\ 0 & 1} \Gamma$.
\end{proposition}

For later use, we recall here the result on the coset decomposition of the generators $T(p)$ and $T(q,q)$.

\begin{lemma}[{\cite[Proposition 5.2.1]{DiamondShurman2005}}]\label{lem:Hecke-left-coset}
	For any prime $p$, we have the coset decomposition
	\begin{align}\label{eq:T(p)-coset}
		\Gamma_0(N) \pmat{1 & 0 \\ 0 & p} \Gamma_0(N) = \begin{cases}
			\displaystyle{\bigsqcup_{j=0}^{p-1} \Gamma_0(N) \beta_j} &\text{if } p \mid N,\\
			\displaystyle{\Gamma_0(N) \beta_\infty \cup \bigsqcup_{j=0}^{p-1} \Gamma_0(N) \beta_j} &\text{if } p \nmid N,
		\end{cases}
	\end{align}
	where $\beta_\infty = \smat{p & 0 \\ 0 & 1}$ and $\beta_j = \smat{1 & j \\ 0 & p}$, $(0 \le j \le p-1)$. For any prime $q \nmid N$, we have
	\[
		\Gamma_0(N) \pmat{q & 0 \\ 0 & q} \Gamma_0(N) = \Gamma_0(N) \pmat{q & 0 \\ 0 & q}.
	\]
\end{lemma}

In the following subsections, we introduce representations of the Hecke algebra for the three spaces. 

\subsection{Hecke operators on divisors}

Let $\GL_2^+(\R)$ denote the subgroup of $\GL_2(\R)$ consisting of matrices with positive determinant. The group $\GL_2^+(\R)$ acts on the upper-half plane $\bbH$ by fractional linear transformations $\smat{a & b \\ c & d} \cdot \tau \coloneqq \frac{a\tau+b}{c\tau+d}$. For the compact Riemann surface $X_0(N) \coloneqq \Gamma_0(N) \backslash \bbH^*$, we let $\Div(X_0(N))$ denote the free $\Z$-module generated by the points of $X_0(N)$, and call its element a \emph{divisor} on $X_0(N)$. Namely, a divisor $D \in \Div(X_0(N))$ is a formal sum of the form
\[
	D = \sum_{[z] \in X_0(N)} n_z [z],
\]
where $n_z \in \Z$ for all $[z] \in X_0(N)$ and $n_z = 0$ for almost all $[z]$. Here, $[z]$ denotes the equivalence class of $z \in \bbH^*$ in $\Gamma_0(N) \backslash \bbH^*$. More precisely, it should be written as $n_{[z]}$ instead of $n_z$, but for simplicity of notation, we will use $n_z$. Now, we extend the coefficients to $\Q$ and consider the abelian group
\[
	\Div(X_0(N))_\Q \coloneqq \Div(X_0(N)) \otimes_\Z \Q.
\]
As is well-known, the Hecke algebra $R_0(N)$ acts on $\Div(X_0(N))_\Q$ as follows.

\begin{definition}
	For any $u = \Gamma_0(N) \alpha \Gamma_0(N) \in R_0(N)$ and a divisor $D = \sum_{[z] \in X_0(N)} n_z [z] \in \Div(X_0(N))_\Q$, we define $u D \in \Div(X_0(N))_\Q$ by
	\begin{align}\label{eq:def-R0-Div}
		u D \coloneqq \sum_{[z] \in X_0(N)} n_z \sum_{i \in I} [\alpha_i z],
	\end{align}
	where $\Gamma_0(N) \alpha \Gamma_0(N) = \bigsqcup_{i \in I} \Gamma_0(N) \alpha_i$ is a coset decomposition. Moreover, for any $u, v \in R_0(N)$, we define $(u + v) D = uD + vD$.
\end{definition}

Since $[z]$ is a point on $X_0(N) = \Gamma_0(N) \backslash \bbH^*$, we have $[\gamma \alpha_i z] = [\alpha_i z]$ for any $\gamma \in \Gamma_0(N)$. This implies that the right-hand side of \eqref{eq:def-R0-Div} is independent of the choice of $\{\alpha_i\}$. For any $u \in R_0(N)$, the map $D \mapsto uD$ defines an element of $\End(\Div(X_0(N))_\Q)$, because it satisfies $u(D_1 + D_2) = uD_1 + uD_2$. 

\begin{proposition}\label{prop:R0-action-Div}
	 The above map $R_0(N) \to \End(\Div(X_0(N))_\Q)$ is a ring-homomorphism.
\end{proposition}

\begin{proof}
	By definition, the map is linear. Therefore, it suffices to verify the homomorphism property with respect to multiplication. Let $u = \Gamma_0(N) \alpha \Gamma_0(N)$ and $v = \Gamma_0(N) \beta \Gamma_0(N)$ be elements of $R_0(N)$, with coset decompositions as in~\eqref{eq:coset-decomp-ex}. Then we have
	\begin{align*}
		(u\cdot v)D &= \sum_{[z] \in X_0(N)} n_z \sum_{(i,j) \in I \times J} [(\alpha_i \beta_j) z]\\
			&= \sum_{[z] \in X_0(N)} n_z \sum_{j \in J} \sum_{i \in I} [\alpha_i (\beta_j z)]\\
			&= u \left(\sum_{[z] \in X_0(N)} n_z \sum_{j \in J} [\beta_j z] \right) = u (vD),
	\end{align*}
	as desired.
\end{proof}

\subsection{Hecke operators on modular forms}\label{sec:Hecke-modualr-cla}

Let $k$ be an integer. For a matrix $\alpha = \smat{a & b \\ c & d} \in \GL_2^+(\R)$, the action of \emph{weight $k$ slash operator} on a meromorphic function $f: \bbH \to \C$ is defined by
\[
	(f|_k \alpha)(\tau) \coloneqq j(\alpha, \tau)^{-k} (\det \alpha)^{k/2} f (\alpha \tau),
\] 
where $j(\alpha, \tau) \coloneqq c\tau+d$ is the \emph{automorphy factor}. For a congruence subgroup $\Gamma$, a meromorphic function $f$ that is also meromorphic at each cusp and satisfies $f|_k \gamma = f$ for any $\gamma \in \Gamma$ is called a \emph{meromorphic modular form} of weight $k$ on $\Gamma$. The $\C$-vector space consisting of such functions is denoted by $\merM_k(\Gamma)$.

\begin{definition}\label{def:classical-Hecke-op}
	For any $u = \Gamma_0(N) \alpha \Gamma_0(N) \in R_0(N)$ and a meromorphic modular form $f \in \merM_k(\Gamma_0(N))$, we define $f|_k u \in \merM_k(\Gamma_0(N))$ by
	\[
		(f|_k u)(\tau) \coloneqq \sum_{i \in I} (f|_k \alpha_i)(\tau),
	\]
	where $\Gamma_0(N) \alpha \Gamma_0(N) = \bigsqcup_{i \in I} \Gamma_0(N) \alpha_i$ is a coset decomposition. Moreover, for any $u, v \in R_0(N)$, we define $f|_k (u+v) = f|_k u + f|_k v$.
\end{definition}

Since $\Gamma_0(N) \alpha \Gamma_0(N)$ is invariant under right multiplication by $\Gamma_0(N)$, the right multiplication by any $\gamma \in \Gamma_0(N)$ permutes $\{\alpha_i\}_{i \in I}$. Thus, we obtain $f|_k u \in \merM_k(\Gamma_0(N))$. For any $u \in R_0(N)$, the map $f \mapsto f|_k u$ defines an element of $\End(\merM_k(\Gamma_0(N)))$, because it satisfies $(f + g)|_k u = f|_k u + g|_k u$.

\begin{proposition}\label{prop:Hecke-ring-hom}
	The above map $R_0(N) \to \End(\merM_k(\Gamma_0(N)))$ is a ring-homomorphism.
\end{proposition}

\begin{proof}
	The proof is similar to that of \cref{prop:R0-action-Div}. For any modular form $f \in \merM_k(N)$, we have
	\begin{align*}
		(f|_k(u \cdot v))(\tau) &= \sum_{(i,j) \in I \times J} (f|_k \alpha_i \beta_j)(\tau)\\
			&= \sum_{i \in I} \sum_{j \in J} j(\alpha_i \beta_j, \tau)^{-k} (\det \alpha_i \beta_j)^{k/2} f((\alpha_i \beta_j)\tau).
	\end{align*}
	Since $j(\alpha_i \beta_j, \tau) = j(\alpha_i, \beta_j \tau) j(\beta_j, \tau)$ holds, we obtain
	\begin{align*}
		&= \sum_{i \in I} \sum_{j \in J} j(\alpha_i, \beta_j \tau)^{-k} j(\beta_j, \tau)^{-k} (\det \alpha_i)^{k/2} (\det \beta_j)^{k/2} f(\alpha_i (\beta_j \tau))\\
		&= \sum_{j \in J} j(\beta_j, \tau)^{-k}(\det \beta_j)^{k/2} (f|_k u) (\beta_j \tau)\\
		&= ((f|_k u)|_k v)(\tau),
	\end{align*}
	which completes the proof.
\end{proof}

In many references, the image of $T(n) \in R_0(N)$ under this map is referred to as the \emph{Hecke operator}. 

For later use, we also define the action $(F|_0u)(\tau)$ for modular forms of weight $0$ (not necessarily meromorphic), that is, for maps $F: X_0(N) \to \C$. This action can be defined in a well-defined manner using the same formula as in \cref{def:classical-Hecke-op}:
\[
	(F|_0 u)(\tau) = \sum_{i \in I} F(\alpha_i \tau).
\]
Moreover, as in \cref{prop:Hecke-ring-hom}, one can verify that $R_0(N) \to \mathrm{End}(\mathrm{Map}(X_0(N), \C))$ is a ring-homomorphism.

\subsection{Hecke operators on multiplicative group of meromorphic modular forms}

We define the ring of integer-weight meromorphic modular forms for $\Gamma_0(N)$ as
\[
	\merM_*(\Gamma_0(N)) \coloneqq \bigcup_{k \in \Z} \merM_k(\Gamma_0(N)).
\]
In particular, $\merM_*(\Gamma_0(N))$ forms an abelian group under the usual multiplication. Here we emphasize this structure and regard $\merM_*(\Gamma_0(N))$ as a multiplicative abelian group.

\begin{definition}
	For any $u = \Gamma_0(N) \alpha \Gamma_0(N) \in R_0(N)$ and a meromorphic modular form $f \in \merM_k(\Gamma_0(N))$, we define $f|_* u \in \merM_*(\Gamma_0(N))$ by
	\[
		(f|_* u)(\tau) \coloneqq \prod_{i \in I} (f|_k \alpha_i) (\tau),
	\]
	where $\Gamma_0(N) \alpha \Gamma_0(N) = \bigsqcup_{i \in I} \Gamma_0(N) \alpha_i$ is a coset decomposition. Moreover, for any $u, v \in R_0(N)$, we define $f|_* (u+v) = f|_* u \cdot f|_* v$. 
\end{definition}

Since $\Gamma_0(N) \alpha \Gamma_0(N)$ is invariant under right multiplication by $\Gamma_0(N)$, the right multiplication by any $\gamma \in \Gamma_0(N)$ permutes $\{\alpha_i\}_{i \in I}$. Thus, we obtain $f|_* u \in \merM_{k|I|}(\Gamma_0(N)) \subset \merM_*(\Gamma_0(N))$. For any $u \in R_0(N)$, the map $f \mapsto f|_* u$ defines an element of $\End(\merM_*(\Gamma_0(N)))$, because it satisfies $(f \cdot g)|_* u = f|_* u \cdot g|_* u$.

\begin{proposition}\label{prop:R0-End-hom}
	The above map $R_0(N) \to \End(\merM_*(\Gamma_0(N)))$ is a ring-homomorphism.
\end{proposition}

\begin{proof}
	The proof is similar to that of \cref{prop:R0-action-Div}. Indeed, for any meromorphic modular form $f \in \merM_k(\Gamma_0(N))$, we have
	\begin{align*}
		(f|_* (u \cdot v))(\tau) &= \prod_{(i, j) \in I \times J} (f|_k \alpha_i \beta_j) (\tau)\\
			&= \prod_{i \in I} \prod_{j \in J} j(\alpha_i \beta_j, \tau)^{-k} (\det \alpha_i \beta_j)^{k/2} f((\alpha_i \beta_j) \tau).
	\end{align*}
	Since $j(\alpha_i \beta_j, \tau) = j(\alpha_i, \beta_j \tau) j(\beta_j, \tau)$ holds, we obtain
	\begin{align*}
		&= \prod_{i \in I} \prod_{j \in J} j(\alpha_i, \beta_j \tau)^{-k} j(\beta_j, \tau)^{-k} (\det \alpha_i)^{k/2} (\det \beta_j)^{k/2} f(\alpha_i (\beta_j \tau))\\
		&= \prod_{j \in J} j(\beta_j, \tau)^{-k |I|} (\det \beta_j)^{k|I|/2} (f|_* u) (\beta_j \tau)\\
		&= \prod_{j \in J} ((f|_* u)|_{k |I|} \beta_j) (\tau) = ((f|_* u)|_* v)(\tau),
	\end{align*}
	which concludes the proof.
\end{proof}

The image $T(n) \in R_0(N)$ under this map is called the \emph{multiplicative Hecke operator} and was introduced by Guerzhoy~\cite{Guerzhoy2006} in 2006 in the process of answering Borcherds' question on explaining the Hecke equivariance of the Borcherds isomorphism. It was later extended to level $N$ by the third author and Shin~\cite{KimShin2025} and so on. Here, we reformulated it as a representation of the Hecke algebra. For instance, the formula for the composition of the multiplicative Hecke operators 
\[
	f|_*T(m)T(n) = \prod_{\substack{d \mid (m,n) \\ (d,N) = 1}} \left(f|_*T \left(\frac{mn}{d^2}\right)\right)^d
\]
given in~\cite[Theorem 1.7]{KimShin2025} follows from \cref{prop:Hecke-property} and \cref{prop:R0-End-hom}.

\section{Hecke equivariance of the divisor map}\label{sec:div-Hecke}

In \cref{sec-Hecke-operator}, we introduced three representations of the Hecke algebra on the spaces $\Div(X_0(N))_\Q$, $\merM_k(\Gamma_0(N))$, and $\merM_*(\Gamma_0(N))$. In this section, we show that the divisor map $\div: \merM_*(\Gamma_0(N)) \to \Div(X_0(N))_\Q$ is Hecke equivariant.

\subsection{The order}\label{sec:divisor-map}

Let $\Gamma$ be a congruence subgroup, and write $X(\Gamma) \coloneqq \Gamma \backslash \bbH^*$ for the associated compact Riemann surface. Before defining the divisor map, we recall the notion of the \emph{order} $\ord_\Gamma(f; [z]) \in \Q$ of a (non-zero) meromorphic modular form $f \in \merM_k(\Gamma)$ at a point $[z] \in X(\Gamma)$, (see~\cite[Chapters 2 and 3]{DiamondShurman2005} for details).

For $z \in \bbH$, we define the \emph{period} $\omega_\Gamma(z)$ of $[z] \in X(\Gamma)$ by
\[
	\omega_\Gamma(z) \coloneqq [\{\pm I\}\Gamma_z : \{\pm I\}] = \begin{cases}
		|\Gamma_z|/2 &\text{if } -I \in \Gamma_z,\\
		|\Gamma_z| &\text{if } -I \not\in \Gamma_z,
	\end{cases}
\]
where $[G: H]$ denotes the index of groups $H \subset G$, and $\Gamma_z$ is the stabilizer subgroup. Let $\ord_z(f) \in \Z$ denote the order of $f$ as a meromorphic function on $\bbH$ at $\tau = z$. Then we define
\[
	\mathrm{ord}_\Gamma(f; [z]) \coloneqq \frac{\mathrm{ord}_z(f)}{\omega_\Gamma(z)}.
\]

As for a cusp $\rho \in \Q \cup \{i\infty\}$, choose $\sigma_\rho \in \SL_2(\Z)$ such that $\rho = \sigma_\rho i\infty$, and define the \emph{width} $h_\Gamma(\rho)$ of the point $[\rho] \in X(\Gamma)$ by
\begin{align*}
	h_\Gamma(\rho) &\coloneqq [\SL_2(\Z)_\rho : \{\pm I\} \Gamma_\rho] = [\SL_2(\Z)_\infty : \{\pm I\} (\sigma_\rho^{-1} \Gamma \sigma_\rho)_\infty].
\end{align*}
 Since $f|_k \sigma_\rho \in \merM_k(\sigma_\rho^{-1} \Gamma \sigma_\rho)$, it has a Fourier expansion of the form
\[
	(f|_k \sigma_\rho)(\tau) = \sum_{n= \nu}^\infty a_n q_{h_\Gamma(\rho)}^n, \quad (a_{\nu} \neq 0, q_h = e^{2\pi i\tau/h}).
\]
We then define $\ord_\Gamma(f; [\rho]) \coloneqq \nu$. Under this definition, the $\ord_{i\infty}(f)$ appearing in \eqref{thm:BKO} can be written as $\ord_{i\infty}(f) = \ord_{\SL_2(\Z)}(f; [i\infty])$.
 
We now show a few basic properties of the order that will be used later.

\begin{lemma}\label{lem:ord-trans}
	Let $G \subset \Gamma$ be congruence subgroups. For $f \in \merM_k(\Gamma) \setminus \{0\}$ and $z \in \bbH^*$, we have
	\[
		\ord_G(f; [z]) = [\{\pm I\} \Gamma_z : \{\pm I\} G_z]\ord_\Gamma(f; [z]).
	\]
\end{lemma}

\begin{proof}
	Case 1: $z \in \bbH$. By definition, we have
	\begin{align*}
		\ord_G(f; [z]) &= \frac{\ord_z(f)}{\omega_G(z)} = \frac{\ord_z(f)}{[\{\pm I\}\Gamma_z : \{\pm I\}]} \frac{[\{\pm I\}\Gamma_z : \{\pm I\}]}{[\{\pm I\}G_z : \{\pm I\}]}\\
			&= [\{\pm I\}\Gamma_z : \{\pm I\}G_z] \ord_\Gamma(f; [z]).
	\end{align*}
	
	Case 2: $\rho \in \Q \cup \{i\infty\}$. Since $f \in \merM_k(\Gamma) \subset \merM_k(G)$, it follows that
	\[
		(f|_k \sigma_\rho)(\tau) = \sum_{n=\ord_\Gamma(f; [\rho])}^\infty a_n q_{h_\Gamma(\rho)}^n = \sum_{n=\ord_G(f; [\rho])}^\infty a_n q_{h_G(\rho)}^n,
	\]
	which implies 
	\[
		\frac{\ord_\Gamma(f; [\rho])}{h_\Gamma(\rho)} = \frac{\ord_G(f; [\rho])}{h_G(\rho)}.
	\]
	Since
	\[
		\frac{h_G(\rho)}{h_\Gamma(\rho)} = \frac{[\SL_2(\Z)_\rho : \{\pm I\} G_\rho]}{[\SL_2(\Z)_\rho : \{\pm I\} \Gamma_\rho]} = [\{\pm I\} \Gamma_\rho : \{\pm I\} G_\rho],
	\]
	we obtain the desired result.
\end{proof}

\begin{lemma}\label{lem:ord-ord}
	For $f \in \merM_k(\Gamma) \setminus \{0\}$ and $z \in \bbH^*$, if $\alpha \in \GL_2^+(\Q)$ satisfies $\alpha^{-1} \Gamma \alpha \subset \SL_2(\Z)$, then we have
	\[
		\ord_{\Gamma}(f; [\alpha z]) = \ord_{\alpha^{-1} \Gamma \alpha}(f|_k\alpha; [z]).
	\]
\end{lemma}

\begin{proof}
	Case 1: $z \in \bbH$. Since $\ord_{\alpha z}(f) = \ord_z(f|_k \alpha)$, we have by definition
	\begin{align*}
		\ord_\Gamma(f; [\alpha z]) = \frac{\ord_{\alpha z}(f)}{\omega_\Gamma(\alpha z)} = \frac{\ord_z(f|_k\alpha)}{[\{\pm I\}(\alpha^{-1} \Gamma \alpha)_z : \{\pm I\}]} \frac{[\{\pm I\}(\alpha^{-1} \Gamma \alpha)_z : \{\pm I\}]}{[\{\pm I\} \Gamma_{\alpha z} : \{\pm I\}]}.
	\end{align*}
	Using the identity $(\alpha^{-1} \Gamma \alpha)_z = \alpha^{-1} \Gamma_{\alpha z} \alpha$, and noting that $|\{\pm I\} \alpha^{-1} \Gamma_{\alpha z} \alpha| = |\{\pm I\} \Gamma_{\alpha z}|$, it follows that the right-hand side equals $\ord_{\alpha^{-1} \Gamma \alpha}(f|_k \alpha; [z])$.
	
	Case 2: $\rho \in \Q \cup \{i\infty\}$. Since $\alpha \sigma_\rho i\infty = \alpha \rho$, the identity
	\begin{align*}
		h_{\alpha^{-1} \Gamma \alpha}(\rho) &= [\SL_2(\Z)_\infty : \{\pm I\} (\sigma_\rho^{-1} (\alpha^{-1} \Gamma \alpha) \sigma_\rho)_\infty]\\
			&= [\SL_2(\Z)_\infty : \{\pm I\} ((\alpha \sigma_\rho)^{-1} \Gamma (\alpha \sigma_\rho))_\infty]
	\end{align*}
	implies that $h \coloneqq h_{\alpha^{-1}\Gamma \alpha}(\rho) = h_\Gamma(\alpha \rho)$. Moreover, using the identity $(f|_k \alpha)|_k \sigma_\rho = f|_k (\alpha \sigma_\rho)$, we obtain
	\begin{align*}
		((f|_k \alpha)|_k \sigma_\rho)(\tau) = \sum_{n=\ord_{\alpha^{-1}\Gamma \alpha}(f|_k\alpha; [\rho])}^\infty a_n q_h^n = \sum_{n=\ord_{\Gamma}(f; [\alpha \rho])}^\infty a_n q_h^n,
	\end{align*}
	which shows that $\ord_{\alpha^{-1}\Gamma \alpha}(f|_k\alpha; [\rho]) = \ord_\Gamma(f; [\alpha \rho])$.
\end{proof}

\subsection{The divisor map}\label{sec:divisor-map}

We now define the divisor map and state the main theorem. Recall that $X_0(N) = X(\Gamma_0(N))$.

\begin{definition}
	We define the \emph{divisor map} $\div : \merM_*(\Gamma_0(N)) \setminus \{0\} \to \Div(X_0(N))_\Q$ by
	\[
		\div(f) \coloneqq \sum_{[z] \in X_0(N)} \ord_{\Gamma_0(N)}(f; [z]) [z].
	\]
\end{definition}

\begin{theorem}\label{thm:Hecke-equivariance}
	For any $n \in \Z_{>0}$ such that $(n,N) = 1$ and a non-zero $f \in \merM_*(\Gamma_0(N))$, we have 
	\[
		T(n) \div(f) = \div(f|_* T(n)).
	\]
\end{theorem}

\begin{example}
	Let $j_1(\tau)$ be the elliptic modular $j$-function. When $N = 1$, we consider $f(\tau) = j_1(\tau) - j_1(i) \in \merM_0(\SL_2(\Z))$ and $T(2) \in R_0(1)$. Since $f(\tau)$ has a unique double zero at $\tau = i$ and a unique simple pole at $\tau = i\infty$, we have
	\[
		\div(f) = [i] - [i\infty].
	\]
	Note that $\ord_{\SL_2(\Z)}(f; [i]) = 2/\omega_{\SL_2(\Z)}(i) = 1$. By~\cref{lem:Hecke-left-coset}, or as computed in \cref{ex:level1}, the double coset defining the Hecke operator $T(2)$ decomposes as
	\[
		\SL_2(\Z) \pmat{1 & 0 \\ 0 & 2} \SL_2(\Z) = \bigsqcup_{j \in \{\infty, 0, 1\}} \SL_2(\Z) \beta_j,
	\]
	with $\beta_\infty = \smat{2 & 0 \\ 0 & 1}, \beta_0 = \smat{1 & 0 \\ 0 & 2}, \beta_1 = \smat{1 & 1 \\ 0 & 2}$. Thus, we obtain
	\begin{align*}
		T(2) \div(f) &= [2i] + \left[\frac{i}{2}\right] + \left[\frac{i+1}{2}\right] - 3[i\infty]\\
			&= 2[2i] + [i] - 3[i\infty],
	\end{align*}
	where we use the identifications $[i/2] = [2i]$ and $[(i+1)/2] = [i]$ in $\SL_2(\Z) \backslash \bbH$.
	
	On the other hand, by definition, we have
	\[
		(f|_* T(2))(\tau) = f(2\tau) f \left(\frac{\tau}{2}\right) f \left(\frac{\tau+1}{2}\right).
	\]
	On the standard fundamental domain of $\SL_2(\Z)$, namely, $\{\tau \in \bbH : -1/2 \le \Re(\tau) \le 0, |\tau| \ge 1\} \cup \{\tau \in \bbH : 0 < \Re(\tau) < 1/2, |\tau| > 1\} \cup \{i\infty\}$, the function $f|_* T(2)$ has a double zero at $\tau = 2i$ due to $f(\tau/2)$, a double zero at $\tau = i$ due to $f((\tau+1)/2)$, and triple pole at $\tau = i\infty$. The factor $f(2\tau)$ has no zero on this domain. 
	Therefore, we obtain
	\[
		\div(f|_* T(2)) = 2[2i] + [i] - 3[i\infty].
	\]
	Thus, we have shown that $T(2) \div(f) = \div(f|_* T(2))$. A noteworthy point here is that, in the computation of $T(2) \div(f)$, the contribution from $\beta_\infty = \smat{2 & 0 \\ 0 & 1}$ yields $[2i]$, whereas in the computation of $\div(f|_* T(2))$, the contribution from $\beta_\infty$, namely $f(2\tau)$, does not affect the coefficient of $[2i]$. \cref{thm:Hecke-equivariance} holds due to an extremely delicate permutation, which is not immediately obvious from the definitions.
\end{example}

\begin{example}
	If an integer $n$ is not coprime to $N$, the equivariance fails. For example, let $N=n=2$ and consider the Hauptmodul $j_{2,1} \in \merM_0(\Gamma_0(2))$ defined by
	\[
		j_{2,1}(\tau) = \frac{\Delta(\tau)}{\Delta(2\tau)} = \frac{1}{q} - 24 + 276q - 2048q^2 + \cdots.
	\]
	Since the function $f(\tau) = j_{2,1}(\tau) - j_{2,1}(i) = j_{2,1}(\tau) - 512 \in \merM_0(\Gamma_0(2))$ has a unique simple zero at $\tau = i$ and a unique simple pole at $\tau = i\infty$, we have $\div(f) = [i] - [i\infty]$. The double coset defining $T(2)$ decomposes as $\Gamma_0(2) \smat{1 & 0 \\ 0 & 2} \Gamma_0(2) = \Gamma_0(2)\beta_0 \sqcup \Gamma_0(2) \beta_1$, so that
	\[
		T(2) \div(f) = \left[\frac{i}{2}\right] + \left[\frac{i+1}{2}\right] - 2[i\infty].
	\]
	On the other hand, by examining the pole of the function $(f|_* T(2))(\tau) = f(\tau/2) f((\tau+1)/2)$ at the cusp $i\infty$, we see that it has a simple pole, hence, the coefficient of $[i\infty]$ in $\div(f|_* T(2))$ is $-1$. More precisely, we can verify that
	\[
		f|_*T(2) = - \frac{\Delta(\tau)}{\Delta(2\tau)} + 286720 + 2097152 \frac{\Delta(2\tau)}{\Delta(\tau)}.
	\]
	In particular, at the cusps, the divisor of $f|_*T(2)$ is given by $-[i\infty] - [0]$. Therefore, $\div(f|_*T(2))$ cannot equal $T(2) \div(f)$.
\end{example}

\subsection{Proof of \cref{thm:Hecke-equivariance}}

For an integer $n$ coprime to $N$, the element $T(n) \in R_0(N)$ can be expressed as a polynomial in terms of $T(p)$ and $T(q,q)$ for primes $p,q$ not dividing $N$. Thus, it suffices to prove the theorem for these generators. From the latter part of~\cref{lem:Hecke-left-coset}, we observe that $T(q,q)$ acts as the identity on both abelian groups. Therefore, the problem reduces to the case of $T(p)$ for any prime $p \nmid N$.

Before beginning the proof, we prepare two elementary lemmas from group theory. The first concerns the fibers of the covering map $X(G) \to X(\Gamma)$ induced by a congruence subgroups $G \subset \Gamma$. Here we state the result under the assumption that $G \triangleleft \Gamma$, as this is the case needed later. In this subsection, we sometimes use notation such as $[z]_\Gamma \in X(\Gamma)$ with subscripts to clarify the space in which the divisor is considered.

\begin{lemma}\label{lem:unfolding}
	Let $G \triangleleft \Gamma$ be congruence subgroups. Fix a complete set of representatives (a fundamental domain) $\mathcal{F}(\Gamma) \subset \bbH^*$ for $\Gamma \backslash \bbH^*$. We define a map $\varphi: G \backslash \Gamma \times \mathcal{F}(\Gamma) \to X(G)$ by
	\[
		\varphi(\alpha, w) \coloneqq [\alpha w]_G.
	\]
	Then, $\varphi$ is surjective, and for any $z \in \bbH^*$, we have $\# \varphi^{-1}([z]_G) = [\Gamma_z : G_z]$. 
\end{lemma}

\begin{proof}
	For any $z \in \bbH^*$, there exist a unique $w \in \mathcal{F}(\Gamma)$ and some $\gamma \in \Gamma$ such that $z = \gamma w$. Moreover, there exists a unique $\alpha \in G \backslash \Gamma$ such that $\gamma \in G \alpha$. Therefore
	\[
		[z]_G = [\alpha w]_G = \varphi(\alpha, w),
	\]
	which shows the surjectivity of $\varphi$.
	
	Next, fix such a $\gamma \in \Gamma$ for a given $z \in \bbH^*$. Then we have
	\[
		\{g \in \Gamma : z = gw\} = \Gamma_z \gamma.
	\]
	Since $G$ is a normal subgroup of $\Gamma$, for any $\gamma_z \in \Gamma_z$, we have $\gamma_z \gamma \in \gamma_z G \alpha = G \gamma_z \alpha$. For $\gamma_z, \gamma'_z \in \Gamma_z$, we have $G \gamma_z \alpha = G\gamma'_z \alpha$ if and only if $\gamma'_z \gamma_z^{-1} \in G \cap \Gamma_z = G_z$. Therefore, $\# \varphi^{-1}([z]_G) = [\Gamma_z : G_z]$ as claimed.
\end{proof}

\begin{lemma}\label{lem:Gamma-pN}
	For a prime $p$ such that $p \nmid N$, the right multiplication of the quotient group $\Gamma(pN) \backslash \Gamma_0(N)$ defines a transitive group action on $\Gamma_0(N) \backslash T(p)$.
\end{lemma}

\begin{proof}
	First, we note that $\Gamma(pN) \triangleleft \Gamma_0(N)$. By \cref{lem:Hecke-left-coset}, we have $\Gamma_0(N) \backslash T(p) = \{\Gamma_0(N) \beta_j : j \in \{\infty, 0, 1, \dots, p-1\}\}$. For any $\beta \in T(p)$, since $\beta \Gamma(pN) \beta^{-1} \subset \Gamma(N) \subset \Gamma_0(N)$, we have $\beta \Gamma(pN) \subset \Gamma_0(N) \beta$. In other words, $\Gamma(pN)$ acts trivially. 
	
	Since $(p,N) = 1$, there exist $b,d \in \Z$ such that $pd-bN = 1$. For any $\beta_j = \smat{1 & j \\ 0 & p}$, let
	\[
		\alpha_j \coloneqq \pmat{p & 0 \\ 0 & 1}^{-1} \pmat{p & b \\ N & d} \beta_j = \pmat{1 & b+j \\ N & jN+dp} \in \Gamma_0(N).
	\]
	Then, this $\alpha_j$ maps $\beta_\infty$ to $\beta_j$.
\end{proof}

Let $f \in \merM_k(\Gamma_0(N))$ be any non-zero meromorphic modular form. To prove the theorem, we calculate a divisor defined by
\[
	D \coloneqq \sum_{[z] \in X(\Gamma(pN))} \mathrm{ord}_{\Gamma(pN)}(f; [z]) \sum_{\beta \in \Gamma_0(N) \backslash T(p)} [\beta z]_{\Gamma_0(N)} \in \Div(X_0(N))_\Q
\]
in two ways. Let us verify that the sum is well-defined. For any $\beta \in T(p)$ with $p \nmid N$, we have $\beta \Gamma(pN) \subset \Gamma_0(N) \beta$, as shown in the proof of \cref{lem:Gamma-pN}. This implies that $[\beta (\gamma z)]_{\Gamma_0(N)} = [\beta z]_{\Gamma_0(N)}$ for any $\gamma \in \Gamma(pN)$.

\subsubsection{$T(p) \div(f)$ on the left-hand side}

Since $\Gamma(pN) \triangleleft \Gamma_0(N)$, we have
\begin{align*}
	D &\overset{\text{\cref{lem:unfolding}}}{=} \sum_{w \in \mathcal{F}(\Gamma_0(N))} \sum_{\alpha \in \Gamma(pN) \backslash \Gamma_0(N)} \frac{\ord_{\Gamma(pN)}(f; [\alpha w])}{[\Gamma_0(N)_{\alpha w} : \Gamma(pN)_{\alpha w}]} \sum_{\beta \in \Gamma_0(N) \backslash T(p)} [\beta \alpha w]_{\Gamma_0(N)}\\
		&\overset{\text{\cref{lem:ord-trans}}}{=} \sum_{w \in \mathcal{F}(\Gamma_0(N))} \sum_{\alpha \in \Gamma(pN)\backslash \Gamma_0(N)} \frac{\ord_{\Gamma_0(N)}(f; [\alpha w])}{[\{\pm I\} \Gamma(pN)_{\alpha w} : \Gamma(pN)_{\alpha w}]} \sum_{\beta \in \Gamma_0(N) \backslash T(p)} [\beta \alpha w]_{\Gamma_0(N)}.
\end{align*}
Noting that $[\alpha w]_{\Gamma_0(N)} = [w]_{\Gamma_0(N)}$ and $[\{\pm I\} \Gamma(pN)_{\alpha w} : \Gamma(pN)_{\alpha w}] = [\{\pm I\} \Gamma(pN) : \Gamma(pN)] = 1$ or $2$, this simplifies to
\[
	= \frac{1}{[\{\pm I\} \Gamma(pN) : \Gamma(pN)]} \sum_{w \in \mathcal{F}(\Gamma_0(N))} \mathrm{ord}_{\Gamma_0(N)}(f; [w]) \sum_{\alpha \in \Gamma(pN) \backslash \Gamma_0(N)} \sum_{\beta \in \Gamma_0(N) \backslash T(p)} [\beta \alpha w]_{\Gamma_0(N)}.
\]
Furthermore, by the transitivity of the $\Gamma(pN) \backslash \Gamma_0(N)$-action (see~\cref{lem:Gamma-pN}), for any $\beta, \beta' \in \Gamma_0(N) \backslash T(p)$, we have
\[
	\#\{\alpha \in \Gamma(pN) \backslash \Gamma_0(N) : \Gamma_0(N) \beta \alpha = \Gamma_0(N) \beta'\} = \frac{|\Gamma(pN) \backslash \Gamma_0(N)|}{|\Gamma_0(N) \backslash T(p)|}.
\]
Thus, we have
\begin{align*}
	D &= \frac{1}{[\{\pm I\} \Gamma(pN) : \Gamma(pN)]} \sum_{w \in \mathcal{F}(\Gamma_0(N))} \mathrm{ord}_{\Gamma_0(N)}(f; [w])\\	
		&\qquad \times \sum_{\beta \in \Gamma_0(N) \backslash T(p)} \sum_{\beta' \in \Gamma_0(N) \backslash T(p)} \frac{|\Gamma(pN) \backslash \Gamma_0(N)|}{|\Gamma_0(N) \backslash T(p)|} [\beta' w]_{\Gamma_0(N)} \\
	&= \frac{[\Gamma_0(N) : \Gamma(pN)]}{[\{\pm I\} \Gamma(pN) : \Gamma(pN)]} \sum_{w \in \mathcal{F}(\Gamma_0(N))} \mathrm{ord}_{\Gamma_0(N)}(f; [w])\sum_{\beta' \in \Gamma_0(N) \backslash T(p)} [\beta' w]_{\Gamma_0(N)}.
\end{align*}
Replacing $w$ with $\gamma w$ for some $\gamma \in \Gamma_0(N)$ merely permutes the inner sum. Therefore, we obtain
\begin{align*}
	D = [\Gamma_0(N) : \{\pm I\} \Gamma(pN)] \sum_{[z] \in X_0(N)} \mathrm{ord}_{\Gamma_0(N)}(f; [z])\sum_{\beta \in \Gamma_0(N) \backslash T(p)} [\beta z]_{\Gamma_0(N)},
\end{align*}
that is,
\begin{align}\label{LHS}
	T(p) \div(f) = \frac{1}{[\Gamma_0(N) : \{\pm I\} \Gamma(pN)]} D.
\end{align}

\subsubsection{$\div(f|_*T(p))$ on the right-hand side}

For each $\beta \in T(p)$, define $\beta' \coloneqq p \beta^{-1}$. Since $p \nmid N$, it follows from \cite[Corollary 6.5.8]{CohenStromberg2017} that there exists a system of representatives $B = \{\beta\}$ such that both $\{\beta\}$ and $\{\beta'\}$ form a system of representatives for the coset decomposition $\Gamma_0(N) \backslash T(p)$. Then, we have
\[	
	D = \sum_{\beta \in B} \sum_{[z] \in X(\Gamma(pN))} \mathrm{ord}_{\Gamma(pN)}(f; [z]) [\beta' z]_{\Gamma_0(N)}.
\]
Changing a variable via $z' = \beta'z$, we obtain by $\beta' \Gamma(pN)z = p \beta^{-1} \Gamma(pN) \beta \beta^{-1} z = \beta^{-1} \Gamma(pN) \beta z'$ that
\begin{align*}
	D &= \sum_{\beta \in B} \sum_{[z'] \in X(\beta^{-1} \Gamma(pN) \beta)} \mathrm{ord}_{\Gamma(pN)}(f; [\beta z']) [z']_{\Gamma_0(N)}\\
		&\overset{\text{\cref{lem:ord-ord}}}{=} \sum_{\beta \in B} \sum_{[z'] \in X(\beta^{-1} \Gamma(pN) \beta)} \ord_{\beta^{-1} \Gamma(pN) \beta} (f|_k \beta; [z']) [z']_{\Gamma_0(N)}.
\end{align*}

By~\cite[Lemma 5.1.1]{DiamondShurman2005}, each $\beta \in T(p)$ determines a congruence subgroup $G_\beta \coloneqq \beta^{-1} \Gamma(pN) \beta \cap \Gamma(pN)$. Hence, we can choose an integer $M > 2$ such that $pN \mid M$ and $\Gamma(M) \subset G_\beta$ for all $\beta \in B$. Moreover, since $\beta^{-1} \Gamma(pN) \beta \subset \Gamma(N)$, it follows that $\Gamma(M) \triangleleft \beta^{-1} \Gamma(pN) \beta$ and $\Gamma(M) \triangleleft \Gamma(pN)$. Therefore, we have
\begin{align*}
	D &\overset{\text{\cref{lem:unfolding}}}{=} \sum_{\beta \in B} \sum_{[z] \in X(\Gamma(M))} \frac{[(\beta^{-1} \Gamma(pN) \beta)_z : \Gamma(M)_z]}{[\beta^{-1} \Gamma(pN) \beta : \Gamma(M)]} \ord_{\beta^{-1} \Gamma(pN) \beta} (f|_k \beta; [z]) [z]_{\Gamma_0(N)}\\
	&\overset{\text{\cref{lem:ord-trans}}}{=} \sum_{\beta \in B} \sum_{[z] \in X(\Gamma(M))} \frac{[(\beta^{-1} \Gamma(pN) \beta)_z : \Gamma(M)_z]}{[\beta^{-1} \Gamma(pN) \beta : \Gamma(M)]} \frac{\ord_{\Gamma(M)} (f|_k \beta; [z])}{[\{\pm I\} (\beta^{-1} \Gamma(pN) \beta)_z : \{\pm I\} \Gamma(M)_z]} [z]_{\Gamma_0(N)}.
\end{align*}
Because we chose $M > 2$, we have $-I \not\in \Gamma(M)$. Hence for any $z$, the stabilizer indices satisfy
\begin{align}\label{eq:ratio-stab}
	\frac{[(\beta^{-1} \Gamma(pN) \beta)_z : \Gamma(M)_z]}{[\{\pm I\} (\beta^{-1} \Gamma(pN) \beta)_z : \{\pm I\} \Gamma(M)_z]} = \frac{2}{[\{\pm I\} \Gamma(pN) : \Gamma(pN)]} = 1 \text{ or } 2.
\end{align}
Since $[\beta^{-1} \Gamma(pN) \beta : \Gamma(M)] = [\Gamma(pN) : \Gamma(M)]$ (see~\cite[Lemma 10.3.1]{CohenStromberg2017}), we get
\begin{align*}
	D &= \frac{2}{[\{\pm I\} \Gamma(pN) : \Gamma(pN)]} \sum_{\beta \in B} \sum_{[z] \in X(\Gamma(M))} \frac{\ord_{\Gamma(M)} (f|_k \beta; [z])}{[\Gamma(pN) : \Gamma(M)]} [z]_{\Gamma_0(N)}.
\end{align*}
By definition, we have $\sum_{\beta \in B} \ord_{\Gamma(M)}(f|_k \beta; [z]) = \ord_{\Gamma(M)}(f|_* T(p); [z])$. Therefore, we obtain
\begin{align*}
	D &= \frac{2}{[\{\pm I\} \Gamma(pN) : \Gamma(M)]} \sum_{[z] \in X(\Gamma(M))} \ord_{\Gamma(M)}(f|_* T(p); [z]) [z]_{\Gamma_0(N)}\\
		&\overset{\text{\cref{lem:unfolding}}}{=} \frac{2}{[\{\pm I\} \Gamma(pN) : \Gamma(M)]} \sum_{w \in \mathcal{F}(\Gamma_0(N))} \sum_{\alpha \in \Gamma(M) \backslash \Gamma_0(N)} \frac{\ord_{\Gamma(M)}(f|_* T(p); [\alpha w])}{[\Gamma_0(N)_{\alpha w}: \Gamma(M)_{\alpha w}]} [\alpha w]_{\Gamma_0(N)}\\
		&\overset{\text{\cref{lem:ord-trans}}}{=} \frac{2}{[\{\pm I\} \Gamma(pN) : \Gamma(M)]} \sum_{w \in \mathcal{F}(\Gamma_0(N))}\\
			&\qquad \qquad \times \sum_{\alpha \in \Gamma(M) \backslash \Gamma_0(N)} \frac{[\Gamma_0(N)_{\alpha w} : \{\pm I\} \Gamma(M)_{\alpha w}]}{[\Gamma_0(N)_{\alpha w}: \Gamma(M)_{\alpha w}]} \ord_{\Gamma_0(N)}(f|_* T(p); [\alpha w]) [\alpha w]_{\Gamma_0(N)}.	
\end{align*}
A few simple calculations yield 
\begin{align*}
	D &= \frac{[\Gamma_0(N) : \Gamma(M)]}{[\{\pm I\} \Gamma(pN) : \Gamma(M)]} \sum_{[z] \in X_0(N)} \ord_{\Gamma_0(N)}(f|_* T(p); [z]) [z]_{\Gamma_0(N)},
\end{align*}
that is,
\begin{align}\label{RHS}
	\div(f|_* T(p)) = \frac{1}{[\Gamma_0(N) : \{\pm I\} \Gamma(pN)]} D.
\end{align}

Combining \eqref{LHS} and \eqref{RHS} completes the proof.

\section{Rohrlich-type divisor sums}\label{sec:Rohrlich}

In our previous work~\cite{JKKM2024}, we studied the Rohrlich-type divisor sum. To express this concept using the divisor map, we introduce the following definition.

\begin{definition}
	For a function $F: X_0(N) \to \C$, we define a map $\mathcal{D}_F: \Div(X_0(N))_\Q \to \C$ by
	\[
		\calD_F(D) \coloneqq \sum_{[z] \in X_0(N)} n_z F(z)
	\]
	for $D = \sum_{[z] \in X_0(N)} n_z [z]$.
\end{definition}

For a function $F: X_0(N) \to \C$ and a meromorphic modular form $f \in \merM_*(\Gamma_0(N))$, the quantity $\calD_F(\div(f))$ is called the \emph{Rohrlich-type divisor sum}. In~\cite{JKKM2024}, we considered the case where $F$ is a weight $0$ smooth modular form, including the $j$-function, and provided its explicit formula in terms of the regularized Petersson inner product. Since the smooth modular form $F$ diverges at a cusp $\rho$ in general, we needed to assign an appropriate value to $F(\rho)$ to regard it as a $\C$-valued function on $X_0(N)$.

In this section, we introduce several consequences that follow from the Hecke equivariance of the divisor map (\cref{thm:Hecke-equivariance}) and the following basic theorem.

\begin{theorem}\label{prop:divisor-sums}
	Let $F: X_0(N) \to \C$. For any $n \in \Z_{>0}$ and $D \in \Div(X_0(N))_\Q$, we have
	\[
		\calD_{F} (T(n) D) = \calD_{F|_0T(n)} (D).
	\]
\end{theorem}

\begin{proof}
	For any prime $p$, we have
	\begin{align*}
		\calD_F(T(p)D) &= \sum_{[z] \in X_0(N)} n_z \sum_{\beta \in \Gamma_0(N) \backslash T(p)} F(\beta z) = \sum_{[z] \in X_0(N)} n_z (F|_0 T(p))(z) = \calD_{F|_0 T(p)} (D),
	\end{align*}
	which implies the desired result.
\end{proof}

\subsection{Niebur--Poincar\'{e} series}

For $s \in \C$ and a non-negative integer $m$, we define $\phi_m(\cdot, s) : \R_{>0} \to \C$ by
\[
	\phi_m(v, s) \coloneqq \begin{cases}
		2\pi \sqrt{m v} I_{s-1/2} (2\pi m v) &\text{if } m > 0,\\
		v^s &\text{if } m = 0,
	\end{cases}
\]
where $I_s(v)$ is the $I$-Bessel function. For $\Re(s) > 1$, the \emph{Niebur--Poincar\'{e} series} $F_{N,-m}(\tau, s)$ is defined by
\[
	F_{N, -m}(\tau, s) \coloneqq \sum_{\gamma \in \Gamma_0(N)_\infty \backslash \Gamma_0(N)} \phi_m(v,s) e^{-2\pi imu} |_0 \gamma.
\]
The stabilizer subgroup $\Gamma_0(N)_\infty$ is generated by $-I$ and $T = \smat{1 & 1 \\ 0 & 1}$, namely, $\Gamma_0(N)_\infty = \{\pm T^n: n \in \Z\}$. In particular, when $m=0$, we write $E_N(\tau, s) \coloneqq F_{N,0}(\tau,s)$ and refer to it as the \emph{Eisenstein series}. This series was studied by Niebur~\cite{Niebur1973}, who showed that it converges absolutely for $\Re(s) > 1$, admits the meromorphic continuation in $s$, and is holomorphic at $s=1$ when $m>0$, (see~\cite[Theorem 5]{Niebur1973}). 

The goal of this subsection is to compute the action of the Hecke operator $T(p)$ on the Niebur--Poincar\'{e} series. In proving the main theorem of this section (\cref{cor:L-Hecke-equiv}), only the case where $p \nmid N$ is necessary. Therefore, we provide a detailed explanation for that case. For the case $p \mid N$, the computation was already carried out by the first three authors in~\cite{JKK2023-Hecke}, so we simply state the results. 

We begin with the following lemma, known as the \emph{intertwining relations}. For instance, it appears in the level $1$ case in~\cite[Chapter 6]{Iwaniec1997}, and although not stated explicitly, it is also used implicitly in the general level $N$ case. For the sake of self-containedness, we provide a full proof in the general setting.

\begin{lemma}\label{lem:Intertwining}
	For any prime $p \nmid N$, define $H_p(N) = \{\beta_\infty\} \cup \{\beta_j : 0 \le j < p\}$ as a complete set of representatives for $\Gamma_0(N) \backslash T(p)$, where $\beta_\infty$ and $\beta_j$ are given in \cref{lem:Hecke-left-coset}. Moreover, we fix
	\[
		G_\infty(N) \coloneqq \left\{ \pmat{* & * \\ cN & d} \in \Gamma_0(N) : (cN, d) = 1 \right\}
	\]
	as a complete set of representatives for $\Gamma_0(N)_\infty \backslash \Gamma_0(N)$. Then the map $H_p(N) \times G_\infty(N) \to \Gamma_0(N)_\infty \backslash T(p)$ defined by $(\alpha, \gamma) \mapsto \Gamma_0(N)_\infty \alpha \gamma$ is bijective. 
\end{lemma}

\begin{proof}
	To show injectivity, we take $\alpha_1, \alpha_2 \in H_p(N)$ and $\gamma_1, \gamma_2 \in G_\infty(N)$ such that $\alpha_1 \gamma_1 = T^n \alpha_2 \gamma_2$ for some $n \in \Z$. Since 
	\[
		\beta_\infty \pmat{* & * \\ cN & d} = \pmat{* & * \\ cN & d}, \quad \beta_j \pmat{* & * \\ cN & d} = \pmat{* & * \\ pcN & pd},
	\]
	it follows that $\gamma_1 = \gamma_2$. Then, $\alpha_1 = T^n \alpha_2$ implies that $\alpha_1 \in \Gamma_0(N) \alpha_2$, and hence $\alpha_1 = \alpha_2$. Therefore, we find $(\alpha_1, \gamma_1) = (\alpha_2, \gamma_2)$.
	
	To prove surjectivity, we show that for any $g \in T(p) = \Gamma_0(N) \smat{1 & 0 \\ 0 & p} \Gamma_0(N)$, there exist $n \in \Z, \alpha \in H_p(N)$, and $\gamma \in \Gamma_0(N)$ such that 
	\begin{align}\label{eq:intertwining}
		T^n g = \alpha \gamma.
	\end{align}
	Assuming this holds, we can choose $m \in \Z$ such that $T^{-m} \gamma \in G_\infty(N)$. Noting the identity
	\[
		\pmat{a & b \\ 0 & d} T^m = T^{m'} \pmat{a & b+am-dm' \\ 0 & d},
	\]
	where $m' \in \Z$ is chosen so that $0 \le b+am-dm' < d$, it follows that for any $\alpha \in H_p(N)$ and $m \in \Z$, there (uniquely) exist $\alpha' \in H_p(N)$ and $m' \in \Z$ such that $\alpha T^m = T^{m'} \alpha'$. Thus, we obtain
	\[
		T^{n-m'} g = T^{-m'} \alpha T^m T^{-m} \gamma = \alpha' \gamma'
	\]
	for some $\alpha' \in H_p(N)$ and $\gamma' \in G_\infty(N)$, proving surjectivity.
	
	We now prove~\eqref{eq:intertwining}. If $g \in \Gamma_0(N) \beta_\infty$, that is, $g = \smat{a & b \\ cN & d} \smat{p & 0 \\ 0 & 1}$ for some $\smat{a & b \\ cN & d} \in \Gamma_0(N)$, then
	\begin{align*}
		\begin{cases}
			\beta_0^{-1} g = \pmat{ap & b \\ cN & d/p} \in \Gamma_0(N) &\text{if } d \equiv 0 \pmod{p},\\
			\beta_\infty^{-1} T^n g = \pmat{a+ncN & (b+dn)/p \\ pcN & d} \in \Gamma_0(N) &\text{for some } n \in \Z, \text{ if } d \not\equiv 0 \pmod{p}.
		\end{cases}
	\end{align*}
	If $g \in \Gamma_0(N) \beta_j$ for some $0 \le j < p$, that is, $g = \smat{a & b \\ cN & d} \smat{1 & j \\ 0 & p}$, then
	\begin{align*}
		\begin{cases}
			\beta_0^{-1} g = \pmat{a & aj + bp \\ cN/p & jcN/p + d} \in \Gamma_0(N) &\text{if } c \equiv 0 \pmod{p},\\
			\beta_\infty^{-1} T^n g = \pmat{(a+ncN)/p & (a+ncN)j/p + b + dn \\ cN & jcN + dp} \in \Gamma_0(N) &\text{for some } n \in \Z, \text{ if } c \not\equiv 0 \pmod{p}.
		\end{cases}
	\end{align*}
	Here, we used the assumption that $p \nmid N$. Thus, the existence of $n \in \Z, \alpha \in H_p(N)$, and $\gamma \in \Gamma_0(N)$ satisfying \eqref{eq:intertwining} in all cases has been confirmed. 
\end{proof}

\begin{proposition}\label{prop:Hecke-Poincare}
	For a non-negative integer $m$ and a prime $p$ with $p \nmid N$, we have
	\[
		(F_{N,-m}|_0 T(p)) (\tau, s) = \begin{cases}
			F_{N, -pm}(\tau, s) + p F_{N,-m/p}(\tau, s) &\text{if } m > 0,\\
			(p^s + p^{1-s}) F_{N,0}(\tau, s) &\text{if } m=0,
		\end{cases}
	\]
	where $F_{N, -m/p}(\tau, s) = 0$ when $m/p \not\in \Z$.
\end{proposition}

\begin{proof}
	By definition, we have
	\begin{align*}
		(F_{N,-m}|_0 T(p)) (\tau, s) &= \sum_{\alpha \in \Gamma_0(N) \backslash T(p)}\left(\sum_{\gamma \in \Gamma_0(N)_\infty \backslash \Gamma_0(N)} \phi_m(v,s) e^{-2\pi imu} |_0 \gamma \right)|_0 \alpha\\
			&= \sum_{g \in \Gamma_0(N)_\infty \backslash T(p)} \phi_m(v,s) e^{-2\pi imu} |_0 g.
	\end{align*}
	By \cref{lem:Intertwining}, this becomes
	\begin{align*}
		&= \sum_{\gamma \in G_\infty(N)} \sum_{\alpha \in H_p(N)} \phi_m(v,s) e^{-2\pi imu} |_0 \alpha \gamma\\
		&= \sum_{\gamma \in G_\infty(N)} \left(\phi_m(pv, s) e^{-2 \pi i m (pu)} + \sum_{j=0}^{p-1} \phi_m \left(\frac{v}{p}, s \right) e^{-2\pi i m \frac{u+j}{p}} \right) |_0 \gamma\\
		&= \sum_{\gamma \in G_\infty(N)} \left(\phi_m(pv, s) e^{-2 \pi i pm u} + p \delta_{p \mid m} \phi_m \left(\frac{v}{p}, s \right) e^{-2\pi i \frac{m}{p}u} \right) |_0 \gamma,
	\end{align*}
	where $\delta_{p \mid m} = 1$ is defined as $1$ if $p \mid m$, and $0$ otherwise.
	
	If $m > 0$, since $\phi_m(xv, s) = \phi_{mx}(v,s)$ for $x \in \Q$ with $mx \in \Z$, we obtain
	\begin{align*}
		(F_{N,-m}|_0 T(p))(\tau, s) = F_{N, -pm}(\tau, s) + p F_{N, -m/p}(\tau, s).
	\end{align*}
	
	If $m = 0$, since $\phi_0(xv,s) = x^s \phi_0(v,s)$, we obtain $(F_{N,0}|_0 T(p))(\tau, s) = (p^s + p^{1-s}) F_{N,0}(\tau, s)$.
\end{proof}

In the case where $p \mid N$, the following is known.

\begin{proposition}[{\cite[Propositions 3.1 and 3.3]{JKK2023-Hecke}}]\label{prop:JKK-pN}
	For a non-negative integer $m$ and a prime $p$ with $p \mid N$, we have
	\begin{align*}
		(F_{N,-m}|_0T(p))(\tau, s) = \begin{cases}
			F_{N, -pm}(\tau,s) + p F_{N/p, -m/p}(\tau, s) - F_{N/p, -m}(p\tau, s) &\text{if } m > 0,\\
			p^s F_{N, 0}(\tau,s) + p^{1-s} F_{N/p, 0}(\tau, s) - F_{N/p, 0}(p\tau, s) &\text{if } m = 0.
		\end{cases}
	\end{align*}
\end{proposition}

\begin{proof}
	First, note that the operator $U_p^*$ in the notation of~\cite{JKK2023-Hecke} coincides with $T(p)$ in this article when $p \mid N$. In \cite[Equations (3.1) and (3.3)]{JKK2023-Hecke}, the following equation was obtained:
	\begin{align*}
		F_{N/p, -m}(p\tau, s) &= \sum_{\gamma \in \Gamma_0(N)_\infty \backslash \Gamma_0(N)} \phi_m(p\Im(\gamma \tau),s) e^{-2\pi i m p \Re(\gamma \tau)}\\
		&\qquad + \begin{cases}
			(F_{N, -m}|_0W_{p,N})(\tau, s)&\text{if } p^2 \nmid N,\\
			0 &\text{if } p^2 \mid N,
		\end{cases}
	\end{align*}
	where $W_{p,N}$ is the Atkin--Lehner involution. In addition, in the proof of~\cite[Proposition 3.3 (2) and (5)]{JKK2023-Hecke}, they showed 
	\begin{align}\label{eq:Hecke-sys-Prop3.3}
	\begin{split}
		(F_{N, -m}|_0 T(p))(\tau, s) &= \sum_{j=0}^{p-1} e^{-2\pi i m j/p} \sum_{\gamma \in \Gamma_0(N/p)_\infty \backslash \Gamma_0(N/p)} \phi_m(\Im(\gamma \tau)/p, s) e^{-2\pi i m/p \Re(\gamma \tau)}\\
			&\qquad  - \begin{cases}
			(F_{N,-m}|_0 W_{p,N})(\tau, s) &\text{if } p^2 \nmid N,\\
			0 &\text{if } p^2 \mid N.
		\end{cases}
	\end{split}
	\end{align}
	By summing these two equations, we have
	\begin{align*}
		&(F_{N, -m}|_0 T(p))(\tau, s) + F_{N/p, -m}(p\tau, s)\\
		&= \sum_{\gamma \in \Gamma_0(N)_\infty \backslash \Gamma_0(N)} \phi_m(pv,s) e^{-2\pi i m p u} |_0 \gamma\\
		&\qquad + \sum_{j=0}^{p-1} e^{-2\pi imj/p} \sum_{\gamma \in \Gamma_0(N/p)_\infty \backslash \Gamma_0(N/p)} \phi_m(v/p, s) e^{-2\pi i mu/p} |_0 \gamma.
	\end{align*}
	From the calculations carried out in the proof of \cref{prop:Hecke-Poincare}, the assertion follows.
\end{proof}

\subsection{$p$-plication formulas for polyharmonic Maass forms}

For use in the next subsection, we introduce notation for the coefficient in the Laurent expansion of the Niebur--Poincar\'{e} series at $s=1$.

\begin{definition}\label{def:Jpoly}
	For a non-negative integer $m$ and an integer $r \in \Z$, we define $\J_{N, m, r}(\tau)$ by
	\[
		\J_{N,m,r}(\tau) \coloneqq \mathrm{Coeff}_{(s-1)^r} F_{N, -m}(\tau, s),
	\]
	that is,
	\begin{align}\label{def:J-Laurent}
		F_{N,-m}(\tau, s) = \sum_{r \in \Z} \J_{N,m,r}(\tau) (s-1)^r.
	\end{align}
\end{definition}

The function $\J_{N,m,r}(\tau)$ is a \emph{polyharmonic Maass form}. This means that it is a smooth modular form annihilated by a finite number of actions of the hyperbolic Laplacian 
\[
	\Delta_0 \coloneqq -v^2 \left(\frac{\partial^2}{\partial u^2} + \frac{\partial^2}{\partial v^2}\right).
\]
In particular, those that vanish after a single action of $\Delta_0$ are called \emph{harmonic Maass forms}. The functions $\J_{N,m,r}(\tau)$ can be shown to satisfy these conditions by following the the original works~\cite{LagariasRhoades2016, Matsusaka2020} where the notion was introduced, (see also~\cite[Section 4.1]{JKKM2024}). However, since this property is not essential to this article, we omit the proof.

\begin{example}
	By the Kronecker limit formula, we have
	\begin{align*}
		\J_{N,0,-1}(\tau) &= \frac{1}{\mathrm{vol}(\Gamma_0(N) \backslash \bbH)},\\
		\J_{1,0,0}(\tau) &= -\frac{1}{2\pi} \log(v^6|\Delta(\tau)|) + C,
	\end{align*}
	where $C$ is an explicit constant, and $\Delta(\tau) = q \prod_{n=1}^\infty (1-q^n)^{24}$ is a cusp form of weight $12$ on $\SL_2(\Z)$. Moreover, 
	\[
		\J_{1,1,0}(\tau) = j_1(\tau) = q^{-1} + 24 + 196884q + \cdots
	\]
	is the elliptic modular $j$-function, and $\J_{N,n,1}(\tau)$ coincides with the function $- \J_{N,n}(\tau)$ introduced in the previous work~\cite[Corollary 4.6]{JKKM2024}. Here, we note that in the previous article, the function $\J_{N,n}(\tau)$ was defined using the Maass--Poincar\'{e} series $P_{N,0,-n}(\tau, s)$, whereas in the present article, the function $\J_{N,n,1}(\tau)$ is defined using the Niebur--Poincar\'{e} series $F_{N,-n}(\tau, s)$. These series are related by $F_{N,-n}(\tau, s) = \Gamma(s) P_{N,0,-n}(\tau, s)$. Furthermore, apart from $\J_{N,0,-1}(\tau)$, we have $\J_{N,m,r}(\tau) = 0$ for any $r < 0$. For more details, see our previous article~\cite[Section 4.1]{JKKM2024}.
\end{example}

\begin{corollary}\label{cor:p-plication}
	For any prime $p \nmid N$ and a positive integer $m$, we have
	\begin{align}\label{eq:p-pricate-non}
		\J_{N,m,r}|_0 T(p) = \J_{N,pm,r} + p \J_{N,m/p,r}.
	\end{align}
	As for $m = 0$, we have
	\[
		\J_{N,0,r}|_0 T(p) = \sum_{j=-1}^r (p + (-1)^{r-j}) \frac{(\log p)^{r-j}}{(r-j)!} \J_{N,0,j}.
	\]
\end{corollary}

\begin{proof}
	It immediately follows from \cref{prop:Hecke-Poincare}, \eqref{def:J-Laurent}, and the expansion
	\[
		p^s + p^{1-s} = \sum_{j=0}^\infty (p + (-1)^j) \frac{(\log p)^j}{j!} (s-1)^j,
	\]
	by comparing the $r$-th Laurent coefficient at $s=1$.
\end{proof}

Even in the case $p \mid N$, the same strategy leads to an expression for $\J_{N,m,r}|_0 T(p)$ from \cref{prop:JKK-pN}. For instance, when $m > 0$ and $p \mid N$, we obtain the following:
\begin{align}\label{eq:p-prication}
	(\J_{N,m,r}|_0 T(p))(\tau) = \J_{N,pm,r}(\tau) + p\J_{N/p, m/p,r}(\tau) - \J_{N/p,m,r}(p\tau).
\end{align}
In particular, noting that $F_{N,-1}|_0T(p) = 0$ by \eqref{eq:Hecke-sys-Prop3.3} when $p^2 \mid N$, we obtain the following result.

\begin{corollary}\label{cor:Hecke-system-p}
	For a prime number $p$, we have
	\[
		\J_{N,p,r}(\tau) = \begin{cases}
			(\J_{N,1,r}|_0 T(p))(\tau) &\text{if } p \nmid N,\\
			(\J_{N,1,r}|_0 T(p))(\tau) + \J_{N/p, 1,r}(p \tau) &\text{if } p \mid N \text{ and } p^2 \nmid N,\\
			\J_{N/p, 1,r}(p\tau)  &\text{if } p^2 \mid N.
		\end{cases}
	\]
\end{corollary}

The formulas \eqref{eq:p-pricate-non} and \eqref{eq:p-prication}, known as the \emph{$p$-plication formulas}, were originally discovered by Conway and Norton~\cite{ConwayNorton1979, Norton1984} and further studied by Koike~\cite{Koike-pre, Koike1988} in the case of Hauptmoduln. Here, we present a natural extension of these formulas to polyharmonic Maass forms. In particular, the case $r=0$ of \cref{cor:Hecke-system-p} complements \cite[Theorem 1.1 (3) and Corrigendum]{BKLOR2018} for $p \mid N$ and $p^2 \nmid N$. Although we provide formulas only for $\J_{N,p,r}(\tau)$ with prime $p$, the general case $\J_{N,n,r}(\tau)$ for any positive integer $n$ can be computed inductively using \eqref{eq:p-pricate-non} and \eqref{eq:p-prication}. Since this argument was already carried out in the case $r=0$ in~\cite{JKK2023-Hecke}, we omit the details here.

\subsection{Rohrlich-type divisor sums}

For a meromorphic modular form $f \in \merM_*(\Gamma_0(N))$ and the Niebur--Poincar\'{e} series $F_{N, -m}(\tau, s)$, we consider
\begin{align*}
	\calR_{N,m}(s; f) &\coloneqq \calD_{F_{N,-m}(\cdot, s)}(\div(f)) = \sum_{[z] \in X_0(N)} \ord_{\Gamma_0(N)}(f; [z]) F_{N,-m}(z,s)
\end{align*}
and its Laurent coefficients $\calR_{N,m,r}(f) \coloneqq \mathrm{Coeff}_{(s-1)^r} \calR_{N,m}(s; f) = \calD_{\J_{N,m,r}}(\div(f))$. For the well-definedness, the value of $F_{N,-m}(\tau,s)$ or $\J_{N,m,r}(\tau)$ at a cusp $[\rho] \in X_0(N)$ must be appropriately defined depending on the context. For example, in~\cite[Corollary 4.5]{JKKM2024}, for $m > 0$ and $r=0$, the value $\J_{N,m,0}(\rho)$ is defined as the constant term in the Fourier expansion of $\J_{N,m,0}|_0 \sigma_\rho$, ensuring that
\begin{align}\label{eq:BKLOR}
	\calR_{N,m,0}(f) = -\mathrm{Coeff}_{q^m} \left(\frac{\Theta f}{f} \right)
\end{align}
holds, where $\Theta \coloneqq \frac{1}{2\pi i} \frac{\dd}{\dd \tau}$. Under an appropriate (and any) definition, we have the following.

\begin{theorem}\label{cor:L-Hecke-equiv}
	For any prime $p \nmid N$, we have
	\[
		\calR_{N,m}(s; f|_* T(p)) = \begin{cases}
			\calR_{N, pm}(s;f) + p \calR_{N,m/p}(s;f) &\text{if } m>0,\\
			(p^s + p^{1-s}) \calR_{N,0}(s;f) &\text{if } m=0.
		\end{cases}
	\]
\end{theorem}

\begin{proof}
	This result follows from a combination of \cref{cor:BKO} and \cref{prop:Hecke-Poincare}. From the discussion so far, we obtain
	\begin{align*}
		\calR_{N,m}(s; f|_*T(p)) &= \calD_{F_{N,-m}}(\div(f|_* T(p)))\\
			&\overset{\text{\cref{thm:Hecke-equivariance}}}{=} \calD_{F_{N,-m}}(T(p) \div(f))\\
			&\overset{\text{\cref{prop:divisor-sums}}}{=} \calD_{F_{N,-m}|_0T(p)} (\div(f)).
	\end{align*}
	Applying \cref{prop:Hecke-Poincare} yields the desired result.
\end{proof}

\begin{example}
	For $m=1$, a prime $p \nmid N$, and $f \in \merM_*(\Gamma_0(N))$, \cref{cor:L-Hecke-equiv} implies that $\calR_{N,1}(s; f|_* T(p)) = \calR_{N,p}(s;f)$. By setting $s=1$, we obtain
	\[
		\mathrm{Coeff}_{q^1} \left(\frac{\Theta (f|_*T(p))}{f|_*T(p)}\right) = \mathrm{Coeff}_{q^p} \left(\frac{\Theta f}{f}\right)
	\]
	from \eqref{eq:BKLOR}. This identity can be viewed as a higher-level generalization of $(j_1|_0 T(n), f)_\mathrm{BKO} = (j_1, f|_*T(n))_\mathrm{BKO}$ as is shown in \cref{cor:BKO}.
\end{example}

\begin{example}\label{ex:self-adjoint}
	In \cite[Corollary 4.6]{JKKM2024}, we considered the case where $f \in \merM_k(\Gamma_0(N))$ has neither zeros nor poles at any cusp, that is, $\ord_{\Gamma_0(N)}(f;[\rho]) = 0$ for all cusps $\rho$, and showed that
	\[
		\calR_{N,m,1}(f) = \frac{1}{2\pi} \langle \J_{N,m,0}, \log(v^{k/2}|f|) \rangle^\mathrm{reg}
	\]
	for $m>0$, where $\langle \cdot, \cdot \rangle^\mathrm{reg}$ denotes the regularized Petersson inner product. This is the situation in which the conditions assumed in the explanation of \eqref{eq:Rohrlich-Petersson} are satisfied. Our \cref{cor:L-Hecke-equiv} then directly yields
	\begin{align*}
		\langle \J_{N,m,0}, \log(v^{(p+1)k/2}|f|_*T(p)|)\rangle^{\mathrm{reg}} = \langle \J_{N, pm, 0} + p \J_{N, m/p,0}, \log(v^{k/2}|f|)\rangle^{\mathrm{reg}}.
	\end{align*}
	Since $\log(v^{(p+1)k/2}|f|_* T(p)|) = \log(v^{k/2} |f|) |_0T(p)$ and $\J_{N, pm, 0} + p \J_{N, m/p,0} = \J_{N,m,0}|_0 T(p)$, this is equivalent to the expression
	\[
		\langle \J_{N,m,0}, \log(v^{k/2} |f|) |_0T(p))\rangle^{\mathrm{reg}} = \langle \J_{N,m,0}|_0 T(p), \log(v^{k/2}|f|)\rangle^{\mathrm{reg}},
	\]
	which expresses the self-adjointness of the Hecke operator, as discussed at the end of \cref{sec:Intro}.
\end{example}

Finally, we reprove \cite[Theorem 4.8]{JKKM2024} using \cref{cor:L-Hecke-equiv}, replacing the original argument based on the self-adjointness of the Hecke operator. Before proceeding with the proof, we introduce some notation.

Let $N = 1$. We use the same setting as in~\cite[Section 4.6]{JKKM2024}. For a negative discriminant $d$, let $f_d(\tau)$ denote the weight $1/2$ weakly holomorphic modular form on $\Gamma_0(4)$ satisfying Kohnen's plus condition. For a positive fundamental discriminant $D > 1$, the generalized Borcherds product $\Psi_D(f_d)$ is given by
\[
	\Psi_D(f_d) = \prod_{Q \in \calQ_{dD/\SL_2(\Z)}} (j(\tau) - j(\alpha_Q))^{\frac{\chi_D(Q)}{|\PSL_2(\Z)_Q|}}
\]
as described in \cite[Section 8]{BruinierOno2010}. Here, $\Psi_D(f_d)$ is a meromorphic modular form of weight 0 with a certain unitary character $\sigma: \SL_2(\Z) \to \C^\times$. Since the abelianization of $\SL_2(\Z)$ is isomorphic to $\Z/12\Z$, it follows that $\sigma^{12}$ is the trivial character and $\Psi_D(f_d)^{12} \in \merM_*(\SL_2(\Z))$. 

We now define the twisted traces of CM values of $E(\tau, s) = F_{1,0}(\tau, s)$. This can also be expressed in terms of our function $\calR_{1,0}(s;f)$.
\begin{align*}
	\mathrm{Tr}_{d,D}(E(\cdot, s)) \coloneqq \sum_{Q \in \calQ_{dD}/\SL_2(\Z)} \frac{\chi_D(Q)}{|\PSL_2(\Z)_Q|} E(\alpha_Q, s) = \frac{1}{12} \calR_{1,0}(s; \Psi_D(f_d)^{12}).
\end{align*}

\begin{corollary}\label{cor:DIT-KM}
	For a prime $p \nmid D$, we have
	\[
		\mathrm{Tr}_{dp^2, D}(E(\cdot, s)) = \left(p^{1-s} + p^s - \left(\frac{d}{p}\right)\right) \mathrm{Tr}_{d,D}(E(\cdot, s)) - p \mathrm{Tr}_{d/p^2, D}(E(\cdot, s)).
	\]
\end{corollary}

This recurrence relation was shown to prove~\cite[Theorem 4.8]{JKKM2024} using the regularized inner product expression of $\mathrm{Tr}_{d,D}(E(\cdot, s))$ together with the self-adjointness of the Hecke operator. This result played a key role in providing a concise alternative proof of the decomposition formula of $\mathrm{Tr}_{d,D}(E(\cdot, s))$ by Duke--Imamo\={g}lu--T\'{o}th~\cite{DukeImamogluToth2011} and Kaneko--Mizuno~\cite{KanekoMizuno2020}. In the present proof, the role of self-adjointness is replaced by the Hecke equivariance of the divisor map. Moreover, while the earlier proof relied on analytic arguments regarding the convergence of the integral defining the regularized inner product, the current approach eliminates the need for such considerations.

\begin{proof}
For a prime $p$, let $T_{1/2}(p^2)$ denote the Hecke operator acting on the space of weight $1/2$ modular forms, (see~\cite[Chapter IV, Section 3]{Koblitz1993}). Since
\[
	pT_{1/2}(p^2) f_d = p f_{d/p^2} + \left(\frac{d}{p}\right) f_d + f_{dp^2},
\]
we obtain that
\begin{align*}
	\mathrm{Tr}_{dp^2, D}(E(\cdot, s)) = \frac{1}{12}\calR_{1,0}(s; \Psi_D(pT_{1/2}(p^2)f_d)^{12}) - \left(\frac{d}{p}\right) \mathrm{Tr}_{d,D}(E(\cdot, s)) - p \mathrm{Tr}_{d/p^2, D}(E(\cdot, s)).
\end{align*}
The various Hecke equivariances used in the proof of \cref{cor:L-Hecke-equiv}, and in particular that of the Borcherds product 
\begin{align}\label{eq:Hecke-Borcherds}
	\Psi_D(p T_{1/2}(p^2) f_d)^{12} = \Psi_D(f_d)^{12}|_* T(p)
\end{align}
for a prime $p \nmid D$, (shown in~\cite{Guerzhoy2006, JeonKangKim2023}), imply that
\begin{align*}
	\frac{1}{12} \calR_{1,0}(s; \Psi_D(pT_{1/2}(p^2)f_d)^{12}) &\overset{\eqref{eq:Hecke-Borcherds}}{=} \frac{1}{12} \calR_{1,0}(s; \Psi_D(f_d)^{12}|_* T(p))\\
		&\overset{\text{\cref{cor:L-Hecke-equiv}}}{=} \frac{1}{12} (p^s + p^{1-s}) \calR_{1,0}(s; \Psi_D(f_d)^{12})\\
		&= (p^s + p^{1-s}) \mathrm{Tr}_{d,D}(E(\cdot, s)),
\end{align*}
which concludes the proof.
\end{proof}

\section*{Funding}

The first author was supported by the Basic Science Research Program through the National Research Foundation of Korea (NRF) funded by the Ministry of Education (2022R1A2C1010487).
The second author was supported by the National Research Foundation of Korea (NRF) funded by the Ministry of Education (NRF-2022R1A2C1007188).
The third author was supported by the National Research Foundation of Korea(NRF) grant funded by the Korea government (MSIT) (RS-2024-00348504).
The fourth author was supported by the MEXT Initiative through Kyushu University's Diversity and Super Global Training Program for Female and Young Faculty (SENTAN-Q) and JSPS KAKENHI (JP21K18141 and JP24K16901). 

\bibliographystyle{amsalpha}
\bibliography{References} 

\end{document}